\documentclass{alggeom}

\usepackage{amsmath,amsthm,amssymb,amsfonts,bbm,amscd,mathrsfs}
\usepackage{tikz,tikz-cd}
\usepackage[colorlinks]{hyperref}

\newtheorem{thm}{Theorem}[section]
\newtheorem{lem}[thm]{Lemma}
\newtheorem{prop}[thm]{Proposition}
\newtheorem{cor}[thm]{Corollary}
\numberwithin{equation}{section}

\newtheorem{thmx}{Theorem}

\theoremstyle{definition}
\newtheorem{rmk}[thm]{Remark}
\newtheorem{defn}[thm]{Definition}
\newtheorem*{question}{Question}
\newtheorem*{rmkk}{Remark}
\newtheorem{claim}[thm]{Claim}

\def\Bun{\mathop{\mathscr{B}un}\nolimits}
\def\Pic{\mathop{\mathscr{P}ic}\nolimits}
\def\pic{\mathop{\mathrm{Pic}}\nolimits}
\def\Spec{\mathop{\mathrm{Spec}}\nolimits}
\def\Bun{\mathop{\mathscr{B}un}\nolimits}
\def\gm{\mathop{{\mathbf G}_m}\nolimits}
\def\dt{\mathop{{\rm dt}}\nolimits}

\DeclareMathOperator*{\bigboxtimes}{\raisebox{-0.63ex}{\scalebox{1.9}{$\boxtimes$}}}

\begin{document}

\title{A stacky approach to identifying the semistable locus of bundles}
\author{Dario Wei\ss mann}
\email{dario.weissmann@uni-due.de}
\address{Fakult\"{a}t f\"{u}r Mathematik, Universit\"{a}t Duisburg-Essen, Universit\"{a}ts-strasse 2, 45141 Essen, Germany}
\author{Xucheng Zhang}
\email{zhangxucheng@mail.tsinghua.edu.cn}
\address{Yau Mathematical Sciences Center, Tsinghua University, Beijing 100084, China}
%
\classification{14D20 (primary), 14D23, 14H60 (secondary)}
\keywords{Bundles over curve; stability condition; stacky approach; schematic moduli space}
\thanks{Dario Wei\ss mann was supported by the DFG-Research Training Group 2553 ``Symmetries and classifying spaces: analytic, arithmetic and derived''.}

\begin{abstract}
We show that the semistable locus is the unique maximal open substack of the moduli stack of principal bundles over a curve that admits a schematic moduli space. For rank $2$ vector bundles it coincides with the unique maximal open substack that admits a separated moduli space, but for higher rank there exist other open substacks that admit separated moduli spaces.
\end{abstract}

\maketitle


\section{Introduction}
Semistable vector bundles over a curve form a proper moduli: the stability condition imposed on vector bundles ``accidentally'' coincides with the one arising from the action of a reductive group on a Quot scheme. 
This allows us to use the machinery of geometric invariant theory (GIT), whose outcome is always quasi-projective, to construct the moduli space. In other words, for vector bundles the stability, a priori, is a sufficient condition to obtain a proper moduli space and there is no obvious reason why this condition is also necessary. Therefore it would be interesting to either prove it is indeed the case, or to find more proper moduli spaces of vector bundles.

Let us state this formally. Let $C$ be a smooth projective connected curve and let $G$ be a connected reductive group. Denote by $\mathscr{B}un_G$ the moduli stack of principal $G$-bundles over $C$. The guiding problem in this paper is to single out all open substacks of $\mathscr{B}un_G$ that admit proper good moduli spaces (in the sense of \cite[Definition 4.1]{MR3237451}).

\vspace{0.7em}

By its very definition a good moduli space is only an algebraic space, not necessarily a scheme. If we require that an open substack of $\mathscr{B}un_G$ admits a schematic good moduli space (even without the properness assumption), then the classical stability condition appears.
\begin{thmx}[(Char. 0 version of Theorem \ref{tba})]\label{1535-2}
Let $C$ be a smooth projective connected curve over an algebraically closed field $k$ of characteristic $0$ and let $G$ be a connected reductive group over $k$. Then the open substack $\mathscr{B}un_G^{ss} \subseteq \mathscr{B}un_G$ of semistable principal $G$-bundles is the unique maximal open substack that admits a schematic good moduli space.
\end{thmx}
This gives a stacky approach to identifying the semistable locus of principal bundles and hence provides an intrinsic way to introduce the classical stability condition. As a result the good moduli space of any open substack of $\mathscr{B}un_G$ containing unstable objects, if it exists, cannot be a scheme. In particular, the good moduli space of simple vector bundles is only an algebraic space, not a scheme (see Corollary \ref{1938-new}). Furthermore, as GIT always produces schematic moduli spaces, Theorem \ref{1535-2} implies that the semistable locus is also the maximal locus to which GIT can be applied to construct a moduli space.

The idea to prove Theorem \ref{1535-2} is simple. A result due to Alper describes open substacks of a smooth stack that admit quasi-projective moduli spaces in terms of line bundles over it (see \cite[Theorem 11.14 (ii)]{MR3237451}). If $G$ is almost simple and simply connected, then there is essentially only one line bundle over $\mathscr{B}un_G$ (see \cite[Theorem 17]{MR1961134}) and this line bundle recovers the classical stability condition (see \cite[Corollary 1.16]{MR3758902}). Via the structure theory of reductive groups employed in \cite[\S 5]{MR2628848} we can reduce the general case to the almost simple and simply connected one.

\vspace{0.7em}

In general a good moduli space need not be a scheme and there could be more possibilities for open substacks admitting good moduli spaces. 
Here we consider the case $G=\mathrm{GL}_n$ and denote by $\mathscr{B}un_n^d$ the moduli stack of vector bundles of rank $n$ and degree $d$ over $C$.

In the rank $2$ case, we again find the classical stability condition under the condition that the good moduli space is separated.
\begin{thmx}[(Char. 0 version of Theorem \ref{1724}, \cite{thesis-zhang} Theorem A (3))]\label{thmA}
Let $C$ be a smooth projective connected curve of genus $g_C>1$ over an algebraically closed field $k$ of characteristic $0$. Then the open substack $\mathscr{B}un_2^{d,ss} \subseteq \mathscr{B}un_2^d$ of semistable vector bundles is the unique maximal open substack that admits a separated good moduli space.
\end{thmx}
In higher rank, we construct an example showing that the conclusion in Theorem \ref{thmA} no longer holds. To be precise we have the following result.
\begin{thmx}[(Theorem \ref{rank-3-main}, \cite{thesis-zhang} Theorem B)]\label{thmB}
Let $C$ be a smooth projective connected curve of genus $g_C>2$ over an algebraically closed field $k$. Then there exists an open dense substack $\mathscr{U} \subseteq \mathscr{B}un_3^2$ that admits a separated non-proper good moduli space and is not contained in the open substack $\mathscr{B}un_3^{2,ss} \subseteq \mathscr{B}un_3^2$ of semistable vector bundles.
\end{thmx}
This example can be generalized to arbitrary higher rank, thus showing that there are many other open substacks of $\mathscr{B}un_n^d$ that admit separated non-proper good moduli spaces, i.e., separatedness alone cannot help us identify the classical stability condition. Moreover, by Theorem \ref{1535-2} the good moduli space of $\mathscr{U} \subseteq \mathscr{B}un_3^2$ in Theorem \ref{thmB} is only an algebraic space, not a scheme. In particular, this moduli space does not come from a GIT construction.

The key ingredient to prove Theorem \ref{thmA} and \ref{thmB} is the result \cite[Theorem A]{MR4665776} characterizing algebraic stacks with separated or proper good moduli spaces. To apply this, we find simple translations of the notions of $\Theta$-reductivity and S-completeness in the case of vector bundles.
\begin{rmkk}
We also observe that the maximality type results Theorem \ref{1535-2} and \ref{thmA} actually hold over an arbitrary base field, see 
Lemma \ref{lem-arbitrary}.
\end{rmkk}
Motivated by our result in rank $2$ and the example in rank $3$, we propose the following more specific question:
\begin{question}
Is $\mathscr{B}un_n^{d,ss} \subseteq \mathscr{B}un_n^d$ (or more general $\mathscr{B}un_G^{ss} \subseteq \mathscr{B}un_G$) the unique maximal open substack that admits a proper good moduli space? 
\end{question}
A positive answer to this question would yield another stacky approach to identifying the semistable locus of vector bundles. Nevertheless, Theorem \ref{1535-2} implies that the answer is yes if the condition ``proper'' is replaced by ``quasi-projective''. Theorem \ref{thmA} answers this question affirmatively in rank $2$ and Theorem \ref{thmB} implies that the condition ``proper'' cannot be weakened to ``separated''.

\vspace{0.7em}

The structure of this paper is as follows. In \S\ref{g-bundles}, we give the proof of Theorem \ref{1535-2}. To make the arguments also work in positive characteristic, we need to introduce a stability notion depending on a line bundle over a stack, as a plain generalization of the one defined by Alper in characteristic $0$ (see \cite[Definition 11.1]{MR3237451}). 
We also provide a direct argument for the vector bundles case in \S\ref{vb-case}. In \S\ref{existence}, we recall the existence criteria \cite[Theorem 5.4]{MR4665776}. A first translation of these conditions for vector bundles, including a simple modular interpretation of S-completeness (see Proposition \ref{prop:S-comple}), is provided. As applications we reprove that the good moduli space of semistable vector bundles is separated and that of simple vector bundles is not. In \S\ref{trans-vb2}, we find more specific consequences of the existence criteria in rank $2$, which are sufficient to prove Theorem \ref{thmA}.
We also give the argument generalizing the maximality type results Theorem \ref{1535-2} and \ref{thmA} to an arbitrary base field in \S \ref{sub:generalizing}.
In \S\ref{eg-vb3}, we construct the promised example in rank $3$, therefore proving Theorem \ref{thmB}.
\subsection*{Notation}
Throughout the paper, unless stated otherwise, we work over an algebraically closed field $k$. By a reductive group we always mean a connected reductive group. Let $C$ be a smooth projective connected curve and let $G$ be a reductive group over $k$. Denote by $\mathscr{B}un_G$ the moduli stack of principal $G$-bundles over $C$ and $\mathscr{B}un_n^d$ the moduli stack of vector bundles of rank $n$ and degree $d$ over $C$.
\section{Schematic moduli spaces}\label{g-bundles}
In this section we prove the following result valid in arbitrary characteristic. Recall that in characteristic $0$, the notion of adequate moduli spaces (in the sense of \cite[Definition 5.1.1]{MR3272912}) agrees with the one of good moduli spaces.
\begin{thm}\label{tba}
Let $C$ be a smooth projective connected curve over an algebraically closed field $k$ and let $G$ be a reductive group over $k$. Then the open substack $\mathscr{B}un_G^{ss} \subseteq \mathscr{B}un_G$ of semistable principal $G$-bundles is the unique maximal open substack that admits a schematic adequate moduli space.
\end{thm}
As mentioned before, the starting point is Alper's description for open substacks of a smooth stack that admit quasi-projective good moduli spaces in terms of line bundles in characteristic $0$. A version of this also holds in positive characteristic if one replaces ``good'' by ``adequate''. To include this statement, we start by reviewing Alper's argument in this setting.

\subsection{Adequate stability}
Let $\mathscr{X}$ be an algebraic stack over a field $k$ and $\mathcal{L}$ be a line bundle over $\mathscr{X}$.
\begin{defn}[(Analogue of \cite{MR3237451}, Definition 11.1 (b))]\label{definition-adequately-semistable}
A geometric point $x: \mathrm{Spec}(\mathbf{K}) \to \mathscr{X}$ is \emph{adequately semistable with respect to} $\mathcal{L}$ if there exists a section $s \in \Gamma(\mathscr{X},\mathcal{L}^{\otimes m})$ for some integer $m>0$ such that $s(x) \neq 0$ and $\mathscr{X}_s:=\{y \in \mathscr{X}: s(y) \neq 0\} \to \mathrm{Spec}(k)$ is adequately affine (in the sense of \cite[Definition 4.1.1]{MR3272912}). Denote by $(\mathscr{X})_{\mathcal{L}}^{a\text{-}ss} \subseteq \mathscr{X}$ the open substack of adequately semistable points with respect to $\mathcal{L}$.
\end{defn}
Analogous to Alper's result, we have a description for open substacks of a smooth stack that admit quasi-projective adequate moduli spaces.
\begin{thm}[(Analogue of \cite{MR3237451}, Theorem 11.14 (ii))]
\label{theorem-adequately-semistable-locally-noetherian}
Let $\mathscr{X}$ be a smooth algebraic stack with affine diagonal over a field $k$. If $\mathscr{U} \subseteq \mathscr{X}$ is an open substack that admits a quasi-projective adequate moduli space, then there exists a line bundle $\mathcal{L}$ over $\mathscr{X}$ such that
$\mathscr{U} \subseteq (\mathscr{X})_{\mathcal{L}}^{a\text{-}ss}$.
\end{thm}
\begin{proof}
Let $\pi:\mathscr{U}\to U$ denote the adequate moduli space. As $U$ is quasi-projective, there exists an ample line bundle
$M$ on $U$ together with global sections $m_1,\dots,m_r$ such that $U$ is covered by the non-vanishing loci $U_{m_i}$.
Let $\mathcal{M}=\pi^{\ast}M$ and $t_i=\pi^{\ast}m_i$. Then we obtain $\mathscr{U}=\bigcup_{i=1}^r\mathscr{U}_{t_i}$.
To conclude, we need to extend $\mathcal{M}$ together with the global sections $t_i$ to $\mathscr{X}$
such that the non-vanishing loci remain unchanged.

We use Zorn's lemma to reduce to the case where $\mathscr{X}$ is noetherian as follows.
Consider the set $\Sigma$ consisting of triples 
\[
(\mathscr{V},\mathcal{L},(s_1,\dots,s_r)),
\]
where $\mathscr{V} \subseteq \mathscr{X}$ is an open substack containing $\mathscr{U}$, 
$\mathcal{L}$ is a line bundle over $\mathscr{V}$ extending $\mathcal{M}$,
and $s_i$ is a global section of $\mathcal{L}$ 
extending $t_i$ such that $\mathscr{V}_{s_i}=\mathscr{U}_{t_i}$ for $i=1,\dots,r$. There is a natural ordering $\leq$ on $\Sigma$ and 
it is easy to see that Zorn's lemma applies to $\Sigma$. Furthermore, 
if we can show that the above data extends in the noetherian case, then a maximal element of $\Sigma$
is the desired extension to $\mathscr{X}$, as $\mathscr{X}$ is locally noetherian.

In this case the proof of \cite[Theorem 11.14 (ii)]{MR3237451} carries over
if one adds the following input: 
adequate affineness is equivalent to affineness for any quasi-separated, 
quasi-compact morphism of algebraic spaces (see \cite[Theorem 4.3.1]{MR3272912}). 
The idea is that - using smoothness - we can always extend a line bundle together with sections from an open substack to the entire stack. Adequate affineness shows that the non-vanishing locus of the extended sections stay the same.
\end{proof}
To get a better understanding of adequate stability, we compare it with some other popular stability notions which are also defined by line bundles.
\begin{lem}\label{1408-1}
For a geometric point $x \in \mathscr{X}(\mathbf{K})$, consider the following statements:
\begin{enumerate}
\item
$x$ is cohomologically semistable with respect to $\mathcal{L}$ in the sense of \cite[Definition 11.1]{MR3237451}, i.e., there exists an integer $m>0$ 
and a section $s \in \Gamma(\mathscr{X},\mathcal{L}^{\otimes m})$ such that $s(x) \neq 0$ and $\mathscr{X}_s \to \mathrm{Spec}(k)$ is cohomologically affine.
\item 
$x$ is adequately semistable with respect to $\mathcal{L}$ in the sense of Definition \ref{definition-adequately-semistable}, i.e., 
there exists an integer $m>0$ 
and a section $s \in \Gamma(\mathscr{X},\mathcal{L}^{\otimes m})$ such that $s(x) \neq 0$ and $\mathscr{X}_s \to \mathrm{Spec}(k)$ is adequately affine.
\item
$x$ is $\Theta$-semistable with respect to $\mathcal{L}$ in the sense of \cite[Definition 1.2]{MR3758902} or \cite[Definition 0.0.4]{DHL2014}, i.e., for all morphisms $f: [\mathbf{A}^1_{\mathbf{K}}/\mathbf{G}_{m,\mathbf{K}}] \to \mathscr{X}$ with $f(1) \cong x$ and $f(0) \ncong x$, we have $\mathrm{wt}_{\mathbf{G}_{m,\mathbf{K}}}(f^*\mathcal{L}) \leq 0$.
\end{enumerate}
Then (i) $\Rightarrow$ (ii) $\Rightarrow$ (iii), i.e., $(\mathscr{X})^{c\text{-}ss}_{\mathcal{L}} \subseteq (\mathscr{X})^{a\text{-}ss}_{\mathcal{L}} \subseteq (\mathscr{X})^{\Theta\text{-}ss}_{\mathcal{L}}$.
\end{lem}
These notions remain unchanged if the line bundle $\mathcal{L}$ is replaced by a positive multiple.
\begin{proof}
A cohomologically affine morphism is adequately affine by \cite[Proposition 4.2.1 (2)]{MR3272912}, so (i) $\Rightarrow$ (ii). To see (ii) $\Rightarrow$ (iii), let $x \in \mathscr{X}(\mathbf{K})$ be a geometric point such that there exists a section $s \in \Gamma(\mathscr{X},\mathcal{L}^{\otimes m})$ for some integer $m>0$ satisfying $s(x) \neq 0$. For any morphism $f: [\mathbf{A}^1_{\mathbf{K}}/\mathbf{G}_{m,\mathbf{K}}] \to \mathscr{X}$ with $f(1) \cong x$ the pulled back section $f^*s \in \Gamma([\mathbf{A}^1_{\mathbf{K}}/\mathbf{G}_{m,\mathbf{K}}],f^*\mathcal{L}^{\otimes m})$ does not vanish at $1$. In particular $\Gamma([\mathbf{A}^1_{\mathbf{K}}/\mathbf{G}_{m,\mathbf{K}}],f^*\mathcal{L}^{\otimes m}) \neq 0$. Since for any line bundle $\mathcal{N}$ over $[\mathbf{A}^1_{\mathbf{K}}/\mathbf{G}_{m,\mathbf{K}}]$ we have (see \cite[1.A.a]{MR3758902})
\[
\Gamma([\mathbf{A}^1_{\mathbf{K}}/\mathbf{G}_{m,\mathbf{K}}],\mathcal{N})=\Gamma(\mathbf{A}^1_{\mathbf{K}},\mathcal{N})^{\mathbf{G}_{m,\mathbf{K}}}=\begin{cases}
\mathbf{K} \cdot a^{-d}e & \text{if } \mathrm{wt}_{\mathbf{G}_{m,\mathbf{K}}}(\mathcal{N})=d \leq 0, \\
0 & \text{if } \mathrm{wt}_{\mathbf{G}_{m,\mathbf{K}}}(\mathcal{N})=d>0,
\end{cases}
\]
where $\Gamma(\mathbf{A}^1_{\mathbf{K}},\mathcal{N})=\mathbf{K}[a] \cdot e$, i.e., $e$ is the unique non-vanishing section (up to a scalar) of the line bundle $\mathcal{N}$ over $\mathbf{A}^1_{\mathbf{K}}$. This implies that $m \cdot \mathrm{wt}_{\mathbf{G}_{m,\mathbf{K}}}(f^*\mathcal{L})=\mathrm{wt}_{\mathbf{G}_{m,\mathbf{K}}}(f^*\mathcal{L}^{\otimes m}) \leq 0$.
\end{proof}
To close this subsection, we note that $\Theta$-stability is not particularly interesting for algebraic spaces as there is no testing morphism at each point, so all points are $\Theta$-semistable with respect to any line bundle. 
Here we can prove a slightly more general result which we need later: line bundles from algebraic spaces do not affect the $\Theta$-semistable locus.
\begin{lem}\label{lb-from}
Let $g: \mathscr{X} \to X$ be a morphism from an algebraic stack to an algebraic space. Given two line bundles $\mathcal{L},\mathcal{L}'$ over $\mathscr{X}$ such that $\mathcal{L}$ descends to $X$, we have 
\[
(\mathscr{X})_{\mathcal{L} \otimes \mathcal{L}'}^{\Theta\text{-ss}}=(\mathscr{X})_{\mathcal{L}'}^{\Theta\text{-ss}}.
\]
In particular, we have $(\mathscr{X})_{\mathcal{L}}^{\Theta\text{-ss}}=\mathscr{X}$.
\end{lem}
\begin{proof}
Let $\mathcal{L}=g^*\mathcal{L}_0$ for some line bundle $\mathcal{L}_0$ over $X$. For any field $\mathbf{K}/k$ and any morphism $f: [\mathbf{A}^1_\mathbf{K}/\mathbf{G}_{m,\mathbf{K}}] \to \mathscr{X}$, we compute that
\begin{align*}
\mathrm{wt}_{\mathbf{G}_{m,\mathbf{K}}}(f^*(\mathcal{L} \otimes \mathcal{L}'))&=
\mathrm{wt}_{\mathbf{G}_{m,\mathbf{K}}}((g \circ f)^*\mathcal{L}_0 \otimes f^*\mathcal{L}')=\mathrm{wt}_{\mathbf{G}_{m,\mathbf{K}}}((g \circ f)^*\mathcal{L}_0)+\mathrm{wt}_{\mathbf{G}_{m,\mathbf{K}}}(f^*\mathcal{L}')\\
&=\mathrm{wt}_{\mathbf{G}_{m,\mathbf{K}}}(f^*\mathcal{L}'),
\end{align*}
since the morphism $[\mathbf{A}^1_\mathbf{K}/\mathbf{G}_{m,\mathbf{K}}] \xrightarrow{f} \mathscr{X} \xrightarrow{g} X$ factors via the good moduli space morphism $[\mathbf{A}^1_\mathbf{K}/\mathbf{G}_{m,\mathbf{K}}] \to \mathrm{Spec}(\mathbf{K})$. 
In particular, we see a geometric point $x \in \mathscr{X}(\mathbf{K})$ is $\Theta$-semistable with respect to $\mathcal{L} \otimes \mathcal{L}'$ if and only if it is $\Theta$-semistable with respect to $\mathcal{L}'$.
\end{proof}
\subsection{Proof of Theorem $\ref{tba}$}
The connected components of $\mathscr{B}un_G$ are indexed by elements in $\pi_1(G)$ (see \cite[Theorem 5.8]{MR3013030}), so we may assume that $\mathscr{U} \subseteq \mathscr{B}un_G^d$ for some $d \in \pi_1(G)$. Note that $\mathscr{B}un_G^{d,ss}$ admits a schematic adequate moduli space (see, e.g. \cite[Theorem 3.2.3 (i)]{MR2657374} or \cite[Corollary 3.4.3]{MR2450609}).

Let $\mathscr{U} \subseteq \mathscr{B}un_G^d$ be an open substack that admits a schematic adequate moduli space $\mathscr{U} \to U$. To conclude we need to
show $\mathscr{U} \subseteq \mathscr{B}un_G^{d,ss}$. Choose an affine open cover $\{U_i\}_{i \in I}$ of $U$. Then $\{\mathscr{U}_i\}_{i\in I}:=\{\mathscr{U} \times_U U_i\}_{i \in I}$ is an open cover of $\mathscr{U}$ and each base change $\mathscr{U}_i \to U_i$ is again an adequate moduli space by \cite[Proposition 5.2.9 (1)]{MR3272912}. Note that $U_i$ is quasi-projective as $\mathscr{U}_i$ is of finite type, see \cite[Theorem 6.3.3]{MR3272912}. Since $\mathscr{B}un_G^d$ is smooth (see \cite[Proposition 4.1]{MR3013030}), 
by Theorem \ref{theorem-adequately-semistable-locally-noetherian} we have 
\[
\mathscr{U}_i \subseteq (\mathscr{B}un^d_G)^{a\text{-}ss}_{\mathcal{L}_i} \text{ for some line bundle } \mathcal{L}_i \text{ over } \mathscr{B}un^d_G.
\]
There is a natural line bundle $\mathcal{L}_{\det}$ over $\mathscr{B}un_G^d$, i.e., the determinant line bundle of cohomology (see, e.g. \cite[1.F.a]{MR3758902}), defined as follows. Via the adjoint representation $\mathrm{Ad}: G \to \mathrm{GL(Lie(G))}$, which defines for any $G$-bundle $\mathcal{P}$ over $C$ its adjoint bundle $\mathrm{Ad}_*\mathcal{P}:=\mathcal{P} \times^G \mathrm{Lie}(G) $, we set
\[
\mathcal{L}_{\det}|_{\mathcal{P}}:=\det H^*(C,\mathrm{Ad}_*\mathcal{P})^{-1}.
\]
In particular, if $G=\mathrm{GL}_n$, then for any vector bundle $\mathcal{E}$ over $C$ we have 
\[\mathcal{L}_{\det}|_{\mathcal{E}}:=\det H^*(C,\mathcal{E}nd(\mathcal{E}))^{-1}.
\]
Since the $\Theta$-semistable locus determined by $\mathcal{L}_{\det}$ is $\mathscr{B}un_G^{d,ss}$ (see \cite[1.E and 1.F]{MR3758902}), the idea to determine $(\mathscr{B}un^d_G)_{\mathcal{L}_i}^{a\text{-}ss}$ is to compare $\mathcal{L}_i$ with $\mathcal{L}_{\det}$.
\subsubsection{The case when $G$ is almost simple and simply connected}
In this case $\mathscr{B}un_G$ is connected and by \cite[Theorem 17]{MR1961134} we have
\[
\mathrm{Pic}(\mathscr{B}un_G) \cong \mathbf{Z}.
\]
To determine the adequately semistable locus $(\mathscr{B}un_G)_{\mathcal{L}}^{a\text{-}ss}$ for a line bundle $\mathcal{L}$ over $\mathscr{B}un_G$, we may assume that $\mathcal{L} \cong \mathcal{L}_{\det}^{\otimes e}$ for some integer $e$. Then we compute that
\begin{align*}
(\mathscr{B}un_G)_{\mathcal{L}}^{a\text{-}ss}&=\begin{cases}
(\mathscr{B}un_G)_{\mathcal{L}_{\det}}^{a\text{-}ss} & \text{if } e>0 \\ 
\emptyset & \text{if } e=0 \\
\left(\mathscr{B}un_G\right)_{\mathcal{L}_{\det}^{-1}}^{a\text{-}ss}  & \text{if } e<0
\end{cases} \\
&\subseteq
\begin{cases}
\left(\mathscr{B}un_G\right)_{\mathcal{L}_{\det}}^{\Theta\text{-}ss} & \text{if } e>0 \\
\emptyset & \text{if } e=0 \\
\left(\mathscr{B}un_G\right)_{\mathcal{L}_{\det}^{-1}}^{\Theta\text{-}ss} & \text{if } e<0
\end{cases} \quad \text{(by Lemma \ref{1408-1})} \\
&=
\begin{cases}
\mathscr{B}un_G^{ss} & \text{if } e>0 \text{ (by \cite[Corollary 1.16]{MR3758902})} \\
\emptyset & \text{if } e \leq 0.
\end{cases}
\end{align*}
Indeed, if $e=0$ we have $\Gamma(\mathscr{B}un_G,\mathcal{O}_{\mathscr{B}un_G})=k$ by \cite[Theorem 5.3.1 (i)]{MR2628848}. Thus, for any non-zero section $s \in \Gamma(\mathscr{B}un_G,\mathcal{O}_{\mathscr{B}un_G})$ 
the morphism $(\mathscr{B}un_G)_s=\mathscr{B}un_G \to \mathrm{Spec}(k)$ is not adequately affine since it is not quasi-compact. This shows that $(\mathscr{B}un_G)_{\mathcal{O}_{\mathscr{B}un_G}}^{a\text{-}ss}=\emptyset$. If $e<0$ we again claim that $(\mathscr{B}un_G)_{\mathcal{L}_{\det}^{-1}}^{\Theta\text{-}ss}=\emptyset$, i.e., for any point $\mathcal{P} \in \mathscr{B}un_G(\mathbf{K})$ for some field $\mathbf{K}/k$ there exists a morphism $f: [\mathbf{A}^1_{\mathbf{K}}/\mathbf{G}_{m,\mathbf{K}}] \to \mathscr{B}un_G$ such that $f(1) \cong \mathcal{P}$ and
\[ 
\mathrm{wt}_{\mathbf{G}_{m,\mathbf{K}}}(f^*\mathcal{L}_{\det}^{-1})=-\mathrm{wt}_{\mathbf{G}_{m,\mathbf{K}}}(f^*\mathcal{L}_{\det})>0.
\]
By \cite[1.F.b]{MR3758902}, cocharacters of $G$ and reductions of $\mathcal{P}$ to the corresponding parabolic subgroups define testing morphisms for $\mathcal{P}$. To construct the desired reduction we fix $T \subseteq B \subseteq G$ a maximal torus and a Borel subgroup. Let $\Phi$ be the set of roots of $G$ with respect to $T$. Let $\Phi^+ \subseteq \Phi$ be the positive system of roots determined by $B$ and $\Delta \subseteq \Phi^+$ be the set of simple roots relative to $\Phi^+$. For each $\alpha \in \Delta$ we have
\begin{itemize}
\item 
a rational cocharacter $\omega_\alpha \in X_*(T)_{\mathbf{Q}}$ defined by $\langle \omega_\alpha,\beta \rangle=\delta_{\alpha\beta}$ for any $\beta \in \Delta$.
\item 
a decomposition $\mathrm{Lie}(G)=\oplus_i \mathrm{Lie}(G)_{\alpha,i}$, where $\mathrm{Lie}(G)_{\alpha,i} \subseteq \mathrm{Lie}(G)$ is the subspace on which $\omega_\alpha$ acts with weight $i$. Note that all such weights are integers and each of these spaces is a representation of $T$. 
\end{itemize}
Choose $n>0$ such that $\lambda:=n\sum_{\alpha \in \Delta} \omega_\alpha \in X_*(T)$. Then the parabolic subgroup $P_\lambda \subseteq G$ defined by $\lambda$ is $B$. By \cite[Proposition 3]{MR1362973} there is a reduction $\mathcal{P}_B$ of $\mathcal{P}$ to a $B$-bundle with 
\begin{equation}\label{negative}
\deg(\mathcal{P}_B/R_u(B) \times^{T} \mathrm{Lie}(G)_{\alpha})<0 \text{ for each } \alpha \in \Phi^+,
\end{equation}
as every element in $\Phi^+$ is a non-negative linear combination of simple roots. 
By \cite[1.F.b]{MR3758902} the cocharacter $\lambda$ and the reduction $\mathcal{P}_B$ of $\mathcal{P}$ defines a morphism $f: [\mathbf{A}^1_{\mathbf{K}}/\mathbf{G}_{m,\mathbf{K}}] \to \mathscr{B}un_G$ with $f(1) \cong \mathcal{P}$ and the computation in \cite[1.F.c]{MR3758902} yields
\begin{align*}
\mathrm{wt}_{\mathbf{G}_{m,\mathbf{K}}}(f^*\mathcal{L}_{\det})&=2n\sum_{\alpha \in \Delta} \sum_{i>0} i \cdot \deg(\mathcal{P}_B/R_u(B) \times^{T} \mathrm{Lie}(G)_{\alpha,i}) \\
&=2n\sum_{\alpha \in \Delta} \sum_{i>0} \sum_{\langle \omega_\alpha,\beta \rangle=i} i \cdot \deg(\mathcal{P}_B/R_u(B) \times^{T} \mathrm{Lie}(G)_{\beta})<0.
\end{align*}
Indeed, any root $\beta \in \Phi$ with $\langle \omega_\alpha,\beta \rangle=i$ is a positive root since when expressing $\beta$ in terms of simple roots one of its coefficients (the coefficient of $\alpha$) is positive. 
We conclude by \eqref{negative}.

Since $\mathscr{U}_i$ is non-empty, it follows that $\mathcal{L}_i \cong \mathcal{L}_{\det}^{\otimes e_i}$ for some integer $e_i>0$ and thus $\mathscr{U}_i \subseteq \mathscr{B}un_G^{ss}$ for each $i$. Hence, $\mathscr{U} \subseteq \mathscr{B}un_G^{ss}$.
\subsubsection{The general case}
Up to central isogeny, a reductive group is a product of almost simple, simply connected groups and a torus. This turns out to be useful as we observe that admitting a (schematic) moduli space is well-behaved under the morphism induced by a central isogeny.
Indeed, for any central isogeny of reductive groups
\[
    1 \to \mu \to \tilde G\to G\to 1 \text{ mapping } e\in\pi_1(\tilde G) \text{ to } d\in\pi_1(G),
\]
the induced morphism $\Bun_{\tilde G}^e \to \Bun_G^d$ is a $\Bun_{\mu}$-torsor (see \cite[Example 5.1.4]{MR2628848}).
\begin{claim}
The structure morphism $\Bun_{\mu} \to \mathrm{Spec}(k)$ is adequately affine.  
\end{claim}
\begin{proof}
Since $\mu$ is of multiplicative type over an algebraically closed field we can write $\mu=\prod_i \mu_{n_i}$ for some tuple of positive integers $(n_i)$. 
Then we have $\Bun_\mu \cong \prod_i \Bun_{\mu_{n_i}}$. Applying \cite[Lemma 2.2.1]{MR2628848} to the short exact sequence $1 \to \mu_{n_i} \to \mathbf{G}_m \xrightarrow{(-)^{n_i}} \mathbf{G}_m \to 1$ yields an isomorphism 
\[
\Bun_{\mu_{n_i}} \cong \Pic^0[n_i],
\]
where $\Pic^0[n_i]$ denotes the $n_i$-torsion substack of the Picard stack $\Pic^0$. In particular, $\Bun_{\mu_{n_i}}$ admits a good moduli space $\pic^0[n_i]$, the $n_i$-torsion subgroup of the Picard variety $\pic^0$. Then the structure morphism $\Bun_{\mu} \to \mathrm{Spec}(k)$ decomposes as
\[
\Bun_{\mu} \cong \prod_i \Bun_{\mu_{n_i}} \xrightarrow{\text{gms}} \prod_i \pic^0[n_i] \to \mathrm{Spec}(k),
\]
where the first morphism is by definition cohomologically affine (in particular adequately affine) and the second morphism is finite (in particular affine) since each $\pic^0[n_i]$ is finite over $k$ (see \cite[\href{https://stacks.math.columbia.edu/tag/03RP}{Tag 03RP} (5)]{stacks-project}).
\end{proof}
Thus, we can show that if an open substack of $\Bun_G^d$ admits a (schematic) adequate moduli space, then so does its preimage in $\Bun_{\tilde G}^e$. To be precise:
\begin{lem}
\label{lemma-torsor-ams}
    Let $\mathscr{G}$ be a group stack over an algebraically closed field $k$.
    Let $\varphi:\mathscr{Y}\to\mathscr{X}$ be a $\mathscr{G}$-torsor in the sense of \cite[Definition 5.1.3]{MR2628848}. If $\mathscr{G}\to \Spec(k)$ is adequately affine, 
    then $\varphi$ is adequately affine. In particular, if an open substack $\mathscr{U}\subseteq \mathscr{X}$ admits a (schematic) 
    adequate moduli space, then so does $\varphi^{-1}(\mathscr{U}) \subseteq \mathscr{Y}$.
\end{lem}
\begin{proof}
    By \cite[Proposition 4.2.1 (4)]{MR3272912} adequately affineness
    can be checked after a faithfully flat base change.
    As a $\mathscr{G}$-torsor is faithfully flat and
    \[
     (\mathrm{act},\mathrm{pr}_2): \mathscr{G}\times_{k} \mathscr{Y} \to \mathscr{Y}\times_{\mathscr{X}} \mathscr{Y}
    \]
    is an isomorphism by definition, the adequately affineness
    of $\varphi$ follows from the assumption on $\mathscr{G}$
    and \cite[Proposition 4.2.1 (6)]{MR3272912} applied to
    \[
    \begin{tikzcd}
        \mathscr{G} \times_k \mathscr{Y} \ar[r] \ar[d] & \mathscr{G} \ar[d] \\
        \mathscr{Y} \ar[r] & \mathrm{Spec}(k) \ar[ul,phantom,"\lrcorner"].
    \end{tikzcd}
    \]
    Moreover, if $\mathscr{U}\subseteq \mathscr{X}$ admits an adequate moduli space $\mathscr{U} \to U$, then 
    by the above argument the $\mathscr{G}$-torsor $\varphi^{-1}(\mathscr{U}) \to \mathscr{U}$
    is adequately affine and so is the composition $\psi: \varphi^{-1}(\mathscr{U})\to \mathscr{U}\to U$.
    Thus, the adequate moduli space of $\varphi^{-1}(\mathscr{U})$ is given
    by $\varphi^{-1}(\mathscr{U}) \to V:=\Spec(\psi_{\ast}\mathcal{O}_{\mathscr{\varphi}^{-1}(\mathscr{U})})$, see \cite[Remark 5.1.2]{MR3272912}, 
    which is affine over $U$.
    Then we conclude that $V$ is a scheme if $U$ is.
\end{proof}
\begin{rmk}
For morphisms between moduli stacks of principal bundles (over a general base) induced by a central isogeny of reductive groups, this result can alternatively be proven using \cite[Lemma A.4]{DA23} and the induced morphism between moduli spaces is even finite in this case. For the general question in the introduction that whether $\mathscr{B}un_G^{ss} \subseteq \mathscr{B}un_G$ is the unique maximal open substack that admits a proper moduli space, this stronger assertion can help reduce it to the case where $G$ is almost simple and simply connected. 
\end{rmk}
We fix some notations for the rest of the subsection.
Let $G$ be a reductive group, $\tilde G$ be the universal cover of the derived subgroup $[G,G]$ of $G$, and $Z^0$ 
be the reduced identity component of the center of $G$. By \cite[Lemma 5.3.2]{MR2628848} and the discussion following it,
there exist an extension of reductive groups 
\[
    1\to \tilde G \to \hat G\xrightarrow{\dt} \mathbf{G}_m\to 1
\]
and an extension $\hat\pi:\hat G \to G$ of $\tilde G \to G$ such that the induced map $\pi_1(\hat G) \to \pi_1(G)$ maps $1 \in \mathbf{Z}=\pi_1(\gm)=\pi_1(\hat G)$ to $d \in \pi_1(G)$. Moreover, we have a central isogeny
\[
    1 \to \mu \to Z^0\times \hat G\xrightarrow{(z^0,\hat g)\mapsto (z^0\cdot \hat\pi (\hat g),\dt(\hat g))} G\times \gm \to 1
\]
such that the induced morphism
    \[
        \varphi: \Bun^0_{Z^0}\times \Bun^1_{\hat G} \to \Bun^d_G \times \mathscr{P}ic^1
    \]
is a $\Bun_{\mu}$-torsor. For any line bundle $L$ of degree $1$ over $C$ we denote by $\Bun_{\hat G,L}$ the stack defined by the fiber product
\[
    \begin{tikzcd}
        \Bun_{\hat G,L} \ar[r,"\lambda_L"] \ar[d] & \Bun_{\hat G}^1 \ar[d,"\dt_{\ast}"]\\
        \Spec(k)\ar[r,"L"'] & \Pic^1 \ar[ul,phantom,"\lrcorner"].
    \end{tikzcd}
\]
By base change we again obtain a $\Bun_{\mu}$-torsor
\[
    \varphi_L: \Bun^0_{Z^0} \times \Bun_{\hat G,L} \to \Bun^d_{G}.
\]
Thus, Theorem \ref{1535-2} for $\mathscr{B}un^d_G$ follows from the analogous statement for $\Bun^0_{Z^0}\times \Bun_{\hat G,L}$ and the fact that the morphism $\varphi_L$ maps semistable objects to semistable objects. To this end we need a suitable stability condition on $\Bun^0_{Z^0} \times \Bun_{\hat G,L}$, in particular on $\Bun_{\hat G,L}$.
\begin{defn}
A geometric point $x \in \Bun_{\hat G,L}(\mathbf{K})$ is \emph{semistable} if $\lambda_L(x) \in \Bun_{\hat G}^{1,ss}(\mathbf{K})$, i.e.,
\[
\Bun_{\hat G,L}^{ss}:=\lambda_L^{-1}\left(\Bun_{\hat G}^{1,ss}\right) \subseteq \Bun_{\hat G,L}.
\]
Similarly, we define
\[
\left(\Bun^0_{Z^0}\times \Bun_{\hat G,L}\right)^{ss}:=\Bun^{0,ss}_{Z^0} \times \Bun_{\hat G,L}^{ss}.
\]
\end{defn}
\begin{rmk}\label{1146-1}
Consequently, the morphism $\varphi_L$ maps semistable objects to semistable objects as
$\varphi$ does so and we have the following commutative diagram
\[
    \begin{tikzcd}
        \Bun^0_{Z^0} \times \Bun_{\hat G,L} \ar[r,"\varphi_L"] \ar[d] & \Bun_{G}^d \times \mathrm{Spec}(k) \ar[d,"\mathrm{id}\times L"]\\
        \Bun^0_{Z^0} \times \Bun_{\hat G}^1 \ar[r,"\varphi"'] & \Bun_{G}^d \times \Pic^1 \ar[ul,phantom,"\lrcorner"].
    \end{tikzcd}
\]
\end{rmk}
The stability condition on $\Bun_{\hat G,L}$ can be described using line bundles. If $\mathcal{L}_{\det}$ is the determinant line bundle of cohomology over $\Bun_{\hat{G}}$ and $\mathcal{L}_{\det,L}$ is its pullback to $\Bun_{\hat{G},L}$, then by \cite[Corollary 1.16]{MR3758902} we have $\Bun_{\hat G}^{1,ss}=(\Bun_{\hat G}^1)^{\Theta\text{-}ss}_{\mathcal{L}_{\det}}$ and we claim that
\begin{equation}\label{0326}
\Bun_{\hat G,L}^{ss}=(\Bun_{\hat G,L})^{\Theta\text{-}ss}_{\mathcal{L}_{\det,L}}.
\end{equation}
This is a consequence of the following result.
\begin{lem}\label{0232}
Let $\mathcal{N}$ be a line bundle over $\Bun_{\hat G}^1$ such that the reduced identity component $\hat Z^0$ of the center of $\hat G$ acts trivially. Then the $\Theta$-semistable locus of $\Bun_{\hat G}^1$ with respect to $\mathcal{N}$
    pulls back to the $\Theta$-semistable locus of $\Bun_{\hat G,L}$ with respect to $\lambda_L^*\mathcal{N}$, i.e.,
    \[
        (\Bun_{\hat G,L})^{\Theta\text{-}ss}_{\lambda_L^*\mathcal{N}} = \lambda_L^{-1}\left((\Bun_{\hat G}^1)^{\Theta\text{-}ss}_{\mathcal{N}}\right).
    \]
\end{lem}
\begin{proof}
Let $\Theta:=[\mathbf{A}^1/\mathbf{G}_m]$ be the quotient stack.
By definition
\[
\lambda_L^{-1}\left((\Bun_{\hat G}^1)^{\Theta\text{-}ss}_{\mathcal{N}}\right) \subseteq (\Bun_{\hat G,L})^{\Theta\text{-}ss}_{\lambda_L^*\mathcal{N}}
\]
as any $\Theta$-testing morphism of a point in $\Bun_{\hat G,L}$ with respect to $\lambda_L^*\mathcal{N}$ is a $\Theta$-testing morphism of its image in $\Bun_{\hat G}$ with respect to $\mathcal{N}$. So it remains to prove the converse. The following Cartesian diagram
\[
\begin{tikzcd}
    \lambda_L:\ \Bun_{\hat G,L} \ar[r,"g"] \ar[d] &\Bun_{\hat G,\cong L} \ar[r,hook,"\lambda_{\cong L}"] \ar[d] & \Bun_{\hat G}^1 \ar[d,"\dt_*"] \\
    L:\ \mathrm{Spec}(k) \ar[r] & \mathrm{B}\mathbf{G}_m \ar[ul,phantom,"\lrcorner"] \ar[r,hook] & \Pic^1 \ar[ul,phantom,"\lrcorner"]
\end{tikzcd}
\]
factorizes the morphism $\lambda_L$ as a closed immersion $\lambda_{\cong L}: \Bun_{\hat G,\cong L} \to \Bun_{\hat G}^1$ followed by a $\mathbf{G}_m$-torsor $g: \Bun_{\hat G,L}\to \Bun_{\hat G,\cong L}$. Since any morphism $\Theta \to \Bun_{\hat G}^1$ mapping $1$ into $\Bun_{\hat G,\cong L}$ factors through $\lambda_{\cong L}$, the $\Theta$-semistable locus of $\mathcal{N}$ pulls back to that of $\lambda_{\cong L}^*\mathcal{N}$. It suffices to consider the $\mathbf{G}_m$-torsor part, i.e., if $x \in \Bun_{\hat G,L}(\mathbf{K})$ is $\Theta$-semistable with respect to $\lambda_L^*\mathcal{N}$, then we show that $g(x) \in \Bun_{\hat G,\cong L}(\mathbf{K})$ is $\Theta$-semistable with respect to $\lambda_{\cong L}^*\mathcal{N}$. That is, for any morphism $f: \Theta \to \Bun_{\hat G,\cong L}$ mapping $1$ to $g(x)$, we show
\[
\mathrm{wt}_{\mathbf{G}_m}(f^*\lambda_{\cong L}^*\mathcal{N}) \leq 0.
\]
Clearly, this holds if $f$ lifts to $\Bun_{\hat G,L}$ and we claim that we can assume this by modifying $f$, with the sign of the weight unchanged. Note that a morphism $f: \Theta \to \Bun_{\hat G,\cong L}$ corresponds to two morphisms $\Theta \to \Bun_{\hat G}^1$ and $\Theta \to \mathrm{B}\mathbf{G}_m$ together with an isomorphism when further mapped to $\Pic^1$, i.e., a $\hat G$-bundle $\mathcal{P}$ over $C \times \Theta$ and a line bundle $\mathcal{M}$ over $\Theta$ together with an isomorphism
\[
\dt_*(\mathcal{P}) \cong \mathrm{pr}_C^*L \otimes \mathrm{pr}_{\Theta}^*\mathcal{M} \text{ over } C \times \Theta.
\]
Then $f$ lifts to $\Bun_{\hat G,L}$ if and only if $\mathcal{M}$ is trivial.

For the first morphism $\Theta \to \Bun_{\hat G}^1$, there is a modification arising from $\hat Z^0$. 
Choose a cocharacter $\sigma: \gm \to \hat Z^0$ such that the composition $\gm \xrightarrow{\sigma} \hat Z^0 \hookrightarrow \hat{G} \xrightarrow{\dt} \gm$ corresponds to some negative integer $\ell < 0$ 
(this is possible since $\hat Z^0$ is a split torus and $\hat G$ is generated by $\hat Z^0$ and the derived subgroup $[\hat G,\hat G]$). This yields a morphism $\Theta \to \mathrm{B}\mathbf{G}_m \xrightarrow{\sigma_*} \mathrm{B}\hat Z^0 \to \Bun_{\hat Z^0}^0$, where the morphism $\mathrm{B}\hat Z^0 \to \Bun_{\hat Z^0}^0$ is given by pull back along the projection $C \times S \to S$ for any test scheme $S$. Then we can modify the original $\Theta \to \Bun_{\hat G}^1$ by coupling it with the above morphism $\Theta \to \Bun_{\hat Z^0}^0$ via $\Bun_{\hat Z^0}^0 \times \Bun_{\hat G}^1 \to \Bun_{\hat G}^1$ induced by multiplication 
$\hat Z^0 \times \hat{G} \to \hat{G}$. The image $f(1)$ does not change under this modification. This yields a $\hat G$-bundle $\mathcal{P}'$ over $C \times \Theta$ with 
\[
\dt_*(\mathcal{P}') \cong \dt_*(\mathcal{P}) \otimes \mathrm{pr}_{\Theta}^*\mathcal{M}^{\otimes \ell}.
\]
As $\hat Z^0$ acts trivially on $\mathcal{N}$ we claim that 
this modification does not change the weight.
Indeed, the composition $\hat Z^0 \hookrightarrow \hat{G} \xrightarrow{\dt} \gm$ is surjective as $\dt$ is. 
Then
$\dt_{\ast}:\Bun_{\hat G}^1 \to \Pic^1$ descends to the rigidifications of $\Bun_{\hat G}^1$ and $\Pic^1$ divided by their central automorphisms such that the following diagram commutes (see \cite[Theorem 5.1.5]{MR2007376})
    \[
    \begin{tikzcd}
        \Bun_{\hat G}^1 \ar[r,"\dt_*"] \ar[d] & \Pic^1 \ar[d]\\
        (\Bun_{\hat G}^1)^{\hat Z^0} \ar[r] & (\Pic^1)^{\mathbf{G}_m} \cong \pic^1.
    \end{tikzcd}
    \]
    Note that as $\hat Z^0$ is a split torus, $\Bun_{\hat G}^1 \to (\Bun_{\hat G}^1)^{\hat Z^0}$ is a relative good moduli space.
    Since $\hat Z^0$ acts trivially on $\mathcal{N}$, it descends to the rigidification $(\Bun_{\hat G}^1)^{\hat Z^0}$ by the analogue of \cite[Theorem 10.3]{MR3237451} and we are done.

For the second morphism $\Theta \to \mathrm{B}\mathbf{G}_m$, we can precompose it with $(-)^{n}: \Theta \to \Theta$, for any positive integer $n>0$. Again this modification does not change the sign of the weight and gives a $\hat G$-bundle $\mathcal{P}''$ over $C \times \Theta$ with
\[
\dt_*(\mathcal{P}'') \cong \dt_*(\mathcal{P}') \otimes \mathrm{pr}_{\Theta}^*\mathcal{M}^{\otimes n}.
\]
Therefore, we can take $n=-(\ell+1)$ to kill the factor $\mathcal{M}$.
\end{proof}
To conclude the proof of Theorem \ref{tba}, it suffices to show 
\begin{prop}\label{claim2}
The semistable locus $(\Bun^0_{Z^0}\times \Bun_{\hat G,L})^{ss} \subseteq \Bun^0_{Z^0}\times \Bun_{\hat G,L}$ is the unique maximal open substack that admits a schematic adequate moduli space.
\end{prop}
Indeed, if $\mathscr{U} \subseteq \mathscr{B}un_G^d$ is an open substack that admits a schematic adequate moduli space, 
then by Lemma \ref{lemma-torsor-ams} so does the preimage
$\varphi_L^{-1}(\mathscr{U}) \subseteq \Bun^0_{Z^0} \times \Bun_{\hat G,L}$. 
By Proposition \ref{claim2}
\[
\varphi_L^{-1}(\mathscr{U}) \subseteq (\Bun^0_{Z^0}\times \Bun_{\hat G,L})^{ss},
\]
and hence 
\[
\mathscr{U}=\varphi_L(\varphi_L^{-1}(\mathscr{U})) \subseteq \varphi_L((\Bun^0_{Z^0}\times \Bun_{\hat G,L})^{ss}) \subseteq \Bun^{d,ss}_{G},
\]
where the first equality holds since $\varphi_L$ is a torsor and the final inclusion holds by Remark \ref{1146-1}.
\begin{proof}[Proof of Proposition $\ref{claim2}$]
If $\mathscr{U} \subseteq \Bun^0_{Z^0} \times \Bun_{\hat G,L}$ is an open substack that admits a schematic adequate moduli space, then we show $\mathscr{U} \subseteq (\Bun^0_{Z^0} \times \Bun_{\hat G,L})^{ss}$. As before we may assume that the adequate moduli space is quasi-projective. Since $\Bun^0_{Z^0} \times \Bun_{\hat G,L}$ is smooth (see \cite[Proposition 4.1 and Corollary 4.2]{MR3013030}), by Theorem \ref{theorem-adequately-semistable-locally-noetherian} there exists a line bundle
$\mathcal{L}$ over $\Bun^0_{Z^0} \times \Bun_{\hat G,L}$ such that 
\[
\mathscr{U} \subseteq (\Bun^0_{Z^0} \times \Bun_{\hat G,L})_{\mathcal{L}}^{a\text{-}ss} \subseteq (\Bun^0_{Z^0} \times \Bun_{\hat G,L})_\mathcal{L}^{\Theta\text{-}ss}.
\]
To proceed we need some information about $\mathcal{L}$. By \cite[p. 57, (11)]{MR2628848} we have an isomorphism
\[
    \Bun_{\hat G,L} \cong \prod_{i=1}^r\Bun_{\hat G_i,L},
\]
where $\hat G_i$ is an extension $1 \to \tilde G_i \to \hat G_i \to \gm \to 1$ and $\tilde G\cong \prod_{i=1}^r \tilde G_i$ is the decomposition into
almost simple and simply connected factors. By \cite[Lemma 2.1.4, Proposition 4.4.7 (i), Corollary 4.4.5, and Theorem 4.2.1 (i)]{MR2628848} we 
can describe the Picard group as follows:
\begin{align*}
    \pic(\Bun^0_{Z^0} \times \Bun_{\hat G,L}) &\cong \pic(\Bun^0_{Z^0}) \oplus \pic(\Bun_{\hat G,L}) \\
    &\cong \pic(\Bun^0_{Z^0}) \oplus \bigoplus_{i=1}^r \pic(\Bun_{\hat{G}_i,L}) \\
    &\cong \pic(\Bun^0_{Z^0}) \oplus \bigoplus_{i=1}^r \mathbf{Z}.
\end{align*}
Let $\mathcal{L}_{\det,i}$ be the determinant line bundle of cohomology over $\Bun_{\hat G_i}$ and $\mathcal{L}_{\det,i,L}$ be its pull back to $\Bun_{\hat G_i,L}$. Recall that both the $\Theta$-semistable and adequately semistable locus of a line bundle remain unchanged after replacing the line bundle by a positive tensor power. After such a replacement for $\mathcal{L}$ and $\mathcal{L}_{\det,L}$ we have
\begin{equation}\label{det-fun2}
\mathcal{L}_{\det,L}=\bigboxtimes_{i=1}^r \mathcal{L}_{\det,i,L}^{\otimes c_i} \text{ and } \mathcal{L}=\mathcal{L}_0 \boxtimes \bigboxtimes_{i=1}^r \mathcal{L}_{\det,i,L}^{\otimes e_i}
\end{equation}
for some line bundle $\mathcal{L}_0$ over $\Bun^0_{Z^0}$ and integers $c_i, e_i$. Note that the $\Theta$-semistable locus of $\mathcal{L}$ is the product of the $\Theta$-semistable loci of $\mathcal{L}_0$ and $\mathcal{L}_{\det,i,L}^{\otimes e_i}$ as we can destabilize separately using a filtration on one component and the trivial filtration on the other components. This means
\[
(\Bun^0_{Z^0}\times \Bun_{\hat G,L})_{\mathcal{L}}^{\Theta\text{-}ss}=(\Bun^0_{Z^0})_{\mathcal{L}_0}^{\Theta\text{-}ss} \times \prod_{i=1}^r (\Bun_{\hat G_i,L})^{\Theta\text{-}ss}_{\mathcal{L}_{\det,i,L}^{\otimes e_i}}.
\]
To obtain a non-empty adequately semistable locus, there is some restriction on $\mathcal{L}_0$.
\begin{lem}\label{wt=0}
Let $\mathbf{T}$ be a split torus. Let $\mathscr{X}$ be a connected algebraic stack that is a $\mathbf{T}$-gerbe over some stack $\overline{\mathscr{X}}$ and $\mathcal{L}$ be a line bundle over $\mathscr{X}$. If $(\mathscr{X})_{\mathcal{L}}^{a\text{-}ss} \neq \emptyset$, 
then the central $\mathbf{T}$-automorphisms act trivially on $\mathcal{L}$.
\end{lem}
\begin{proof}
Let $x \in (\mathscr{X})_{\mathcal{L}}^{a\text{-}ss}(\mathbf{K})$ be a geometric point, i.e., there exists a section $s \in \Gamma(\mathscr{X},\mathcal{L}^{\otimes m})$ for some integer $m>0$ with $s(x) \neq 0$. 
Consider the morphism 
$f: \mathrm{B}\mathbf{T}_{\mathbf{K}} \hookrightarrow \mathrm{B}\mathrm{Aut}_{\mathscr{X}(\mathbf{K})}(x) \xrightarrow{x} \mathscr{X}$.
Then we have $f^*s \in \Gamma(\mathrm{B}\mathbf{T}_{\mathbf{K}},f^*\mathcal{L}^{\otimes m}) \neq 0$. 
However, for any line bundle $\mathcal{N}$ over $\mathrm{B}\mathbf{T}_{\mathbf{K}}$
\[
\Gamma(\mathrm{B}\mathbf{T}_{\mathbf{K}},\mathcal{N})=\Gamma(\mathrm{Spec}(\mathbf{K}),\mathcal{N})^{\mathbf{T}_{\mathbf{K}}}=\begin{cases}
    \mathbf{K} \cdot e & \text{if } \mathrm{wt}_{\mathbf{T}_{\mathbf{K}}}(\mathcal{N}) \text{ is trivial} \\
    0 & \text{otherwise,}
\end{cases}
\]
where $\Gamma(\mathrm{Spec}(\mathbf{K}),\mathcal{N})=\mathbf{K} \cdot e$ and $e$ is the unique non-vanishing section (up to scalar) of $\mathcal{N}$. In particular, we see $\mathrm{wt}_{\mathbf{T}_{\mathbf{K}}}(f^*\mathcal{L}^{\otimes m})$ is trivial, i.e., the central $\mathbf{T}$-automorphisms act trivially on the fiber $\mathcal{L}_x$. Since $\mathscr{X}$ is connected and the $\mathbf{T}$-weight is locally constant in a family, we are done.
\end{proof}
\begin{claim}\label{L0}
    $(\Bun^0_{Z^0})_{\mathcal{L}_0}^{\Theta\text{-}ss}=\Bun^0_{Z^0}=\Bun^{0,ss}_{Z^0}$.
\end{claim}
\begin{proof}
We fix an isomorphism $Z^0 \cong \mathbf{G}_m^s$. Then $\Bun^0_{Z^0} \cong \prod_{i=1}^s \Pic^0$ is a trivial $\mathbf{G}_m^s$-gerbe over the product $\prod_{i=1}^s \pic^0$ of Picard varieties. In particular
\[
\Bun^{0,ss}_{Z^0}=\Bun^0_{Z^0}.
\]
As the adequately semistable locus of $\mathcal{L}$ is non-empty, the central $Z^0$-automorphisms act trivially on $\mathcal{L}$ by Lemma \ref{wt=0}.
By definition the same holds for each $\mathcal{L}_{\det,i,L}$. Thus, the central $Z^0$-automorphisms act trivially on $\mathcal{L}_0$ as well.
By \cite[Theorem 10.3]{MR3237451},
the line bundle $\mathcal{L}_0$ descends to the good moduli space $\Bun^0_{Z^0} \to \prod_{i=1}^s \mathrm{Pic}^0$. By Lemma \ref{lb-from} such 
a line bundle does not contribute to the $\Theta$-semistable locus.
\end{proof}
Note that $(\Bun_{\hat G_i,L})^{\Theta\text{-}ss}_{\mathcal{L}_{\det,i,L}^{\otimes e_i}} \neq \emptyset$ for each $i$ as $(\Bun_{Z^0}^0 \times \Bun_{\hat G,L})^{\Theta\text{-}ss}_{\mathcal{L}}\neq \emptyset$ and \eqref{det-fun2}. Furthermore, by Lemma \ref{0232} we have
\[
(\Bun_{\hat G_i,L})_{\mathcal{L}_{\det,i,L}}^{\Theta\text{-}ss}=\lambda_L^{-1}\left((\Bun_{\hat G_i})_{\mathcal{L}_{\det,i}}^{\Theta\text{-}ss}\right).
\]
As in the almost simple and simply connected case we have $e_i > 0$ for each $i$. Similarly, we also have $c_i > 0$ for each $i$ as $(\Bun_{\hat G,L})^{\Theta\text{-}ss}_{\mathcal{L}_{\det,L}} \neq \emptyset$. Finally we compute
\begin{align*}
(\Bun^0_{Z^0}\times \Bun_{\hat G,L})_{\mathcal{L}}^{\Theta\text{-}ss}&=
(\Bun^0_{Z^0})_{\mathcal{L}_0}^{\Theta\text{-}ss} \times \prod_{i=1}^r (\Bun_{\hat G_i,L})^{\Theta\text{-}ss}_{\mathcal{L}_{\det,i,L}^{\otimes e_i}} \\
&=\Bun^{0,ss}_{Z^0} \times \prod_{i=1}^r (\Bun_{\hat G_i,L})^{\Theta\text{-}ss}_{\mathcal{L}_{\det,i,L}} \quad \text{(Claim \ref{L0})}\\
&=\Bun^{0,ss}_{Z^0} \times (\Bun_{\hat G,L})_{\mathcal{L}_{\det,L}}^{\Theta\text{-}ss} \quad \text{(by \eqref{det-fun2})} \\
&=\Bun^{0,ss}_{Z^0} \times \Bun_{\hat G,L}^{ss} \quad \text{(by \eqref{0326})} \\
&=(\Bun^0_{Z^0} \times \Bun_{\hat G,L})^{ss} \quad \text{(by definition)}.
\end{align*}
This shows that $\mathscr{U} \subseteq (\Bun^0_{Z^0} \times \Bun_{\hat G,L})^{ss}$ and we are done.
\end{proof}
\subsubsection{The vector bundles case}\label{vb-case}
For readers only interested in vector bundles, i.e., $G=\mathrm{GL}_n$, there is an alternative argument. Similarly as before, it suffices to show any line bundle $\mathcal{L}$ over $\mathscr{B}un^d_n$ such that $(\mathscr{B}un_n^d)_{\mathcal{L}}^{a\text{-}ss} \neq \emptyset$ satisfies $(\mathscr{B}un_n^d)_{\mathcal{L}}^{a\text{-}ss} \subseteq \mathscr{B}un_n^{d,ss}$. The claim is clear for $n=1$. Assume $n \geq 2$ in the following. By \cite[Theorem 5.3.1 (iv)]{MR2628848}, the morphism $\det: \mathrm{GL}_n \to \mathbf{G}_m$ induces a commutative diagram of commutative groups:
\[
\begin{tikzcd}
0 \ar[r] & \mathrm{Hom}(\pi_1(\mathbf{G}_m),J_C) \ar[r] \ar[d,"\det_*"'] & \mathrm{Pic}(\mathscr{P}ic^d) \ar[r] \ar[d,"\det^*"] & \mathrm{NS}(\mathscr{P}ic^d) \ar[r] \ar[d] & 0\\
0 \ar[r] & \mathrm{Hom}(\pi_1(\mathrm{GL}_n),J_C) \ar[r] & \mathrm{Pic}(\mathscr{B}un^d_n) \ar[r] \ar[d] & \mathrm{NS}(\mathscr{B}un^d_n) \ar[r] & 0\\
 &  & \mathrm{Pic}(\mathscr{B}un_{n,L}) \ar[d,equal,"\text{\cite[Prop 4.2.3 and Thm 4.2.1 (i)]{MR2628848}}"] &  & \\
 & & \mathbf{Z},
\end{tikzcd}
\]
where
\begin{itemize}
\item 
the left-most column is an isomorphism induced by $\det_*: \pi_1(\mathrm{GL}_n) \xrightarrow{\sim} \pi_1(\mathbf{G}_m)$, where $J_C$ denotes the Jacobian of $C$.
\item 
the middle column is exact, induced by the Cartesian diagram
\[
\begin{tikzcd}
\mathscr{B}un_{n,L} \ar[r] \ar[d] & \mathscr{B}un_n^d \ar[d,"\det"] \\
\mathrm{Spec}(k) \ar[r,"L"'] & \mathscr{P}ic^d, \ar[ul,phantom,"\lrcorner"]
\end{tikzcd}
\]
where $L$ is a fixed line bundle of degree $d$ over $C$. The morphism $\mathrm{Pic}(\mathscr{B}un^d_n) \to \mathrm{Pic}(\mathscr{B}un_{n,L})$ is surjective by \cite[Theorem 3.1]{MR2925604}.
\item 
the right-most column is injective by \cite[Proposition 5.2.11]{MR2628848}, where $\mathrm{NS}(-)$ denotes the N\'{e}ron-Severi group.
\end{itemize}
Applying the Snake Lemma yields a short exact sequence
\[
0 \to \mathrm{Pic}(\mathscr{P}ic^d) \xrightarrow{\mathrm{det}^*} \mathrm{Pic}(\mathscr{B}un_n^d) \to \mathrm{Pic}(\mathscr{B}un_{n,L}) \to 0.
\]
Further, it is split since the morphism $\det: \mathscr{B}un_n^d \to \mathscr{P}ic^d$ has a section $\mathcal{N} \mapsto \mathcal{N} \oplus \mathcal{O}_C^{\oplus n-1}$. Then for any line bundle $\mathcal{L}$ over $\mathscr{B}un_n^d$, replacing it by some positive multiple we may assume
\[
\mathcal{L} \cong {\det}^* \mathcal{L}_0 \otimes \mathcal{L}_{\det}^{\otimes e} \text{ for some line bundle } \mathcal{L}_0 \text{ over } \mathscr{P}ic^d \text{ and integer } e.
\]
Note that the stack $\mathscr{B}un_n^d$ is a $\mathbf{G}_m$-gerbe over the rigidified stack $\overline{\mathscr{B}un}_n^d:=[\mathscr{B}un_n^d/\mathrm{B}\mathbf{G}_m]$ obtained by dividing all automorphisms by $\mathbf{G}_m$ (see \cite[Theorem 5.1.5]{MR2007376}).
By Lemma \ref{wt=0} the central $\mathbf{G}_m$-automorphisms act trivially on $\mathcal{L}$. Since in the definition of $\mathcal{L}_{\det}$ we use the determinant of $H^*(C,\mathcal{E}nd(\mathcal{E}))$ instead of $H^*(C,\mathcal{E})$, the central $\mathbf{G}_m$-automorphisms act trivially on $\mathcal{L}_{\det}$ and hence trivially on $\mathcal{L}_0$. Then $\mathcal{L}_0$ descends to the good moduli space $g: \mathscr{P}ic^d \to \mathrm{Pic}^d$ (see \cite[Theorem 10.3]{MR3237451}). Finally, we obtain
\[
\mathcal{L} \cong (g \circ \mathrm{det})^*\mathcal{N} \otimes \mathcal{L}_{\det}^{\otimes e} \text{ for some line bundle } \mathcal{N} \text{ over }\mathrm{Pic}^d.
\]
First we exclude the case $e=0$.
\begin{claim}
If $(\mathscr{B}un_n^d)_{\mathcal{L}}^{a\text{-}ss} \neq \emptyset$, then $e \neq 0$.
\end{claim}
\begin{proof}
If $e=0$, then $\mathcal{L} \cong (g \circ \det)^*\mathcal{N}$. Let $L \in \mathrm{Pic}^d(k)$ be a closed point such that the fiber product $\mathscr{X}:=(\mathscr{B}un_n^d)_{\mathcal{L}}^{a\text{-}ss} \times_{\mathrm{Pic}^d} \mathrm{Spec}(k)$ is non-empty, i.e.,
\[
\begin{tikzcd}
\mathscr{X} \ar[r] \ar[d,hook] & \mathscr{B}un_{n,\cong L} \ar[d,hook] \ar[rr] &  \ & \mathrm{Spec}(k) \ar[d,hook,"L"] \\
(\mathscr{B}un_n^d)_{\mathcal{L}}^{a\text{-}ss} \ar[r] & \mathscr{B}un_n^d \ar[r,"\det"'] \ar[ul,phantom,"\lrcorner"] &  \mathscr{P}ic^d \ar[r,"g"'] & \ar[ul,phantom, "\lrcorner"] \mathrm{Pic}^d.
\end{tikzcd}
\]
By \cite[Lemma 5.2.11]{MR3272912} we have
\[
\emptyset \neq \mathscr{X} \subseteq 
(\mathscr{B}un_{n,\cong L})^{a\text{-}ss}_{\mathcal{L}'} \text{ for } \mathcal{L}':=\mathcal{L}|_{\mathscr{B}un_{n,\cong L}}.
\]
Since $\mathcal{L}$ is pulled back from $\mathrm{Pic}^d$, it follows that $\mathcal{L}' \cong \mathcal{O}_{\mathscr{B}un_{n,\cong L}}$. Note that since $\mathscr{B}un_{n,\cong L}$ is connected, reduced, and satisfies the existence part of the valuative criterion for properness, every global section of $\mathcal{O}_{\mathscr{B}un_{n,\cong L}}$ is constant (see also, e.g. \cite[Proposition 4.4.7 (i)]{MR2628848}). This implies that $(\mathscr{B}un_{n,\cong L})^{a\text{-}ss}_{\mathcal{L}'}=\emptyset$, since $\mathscr{B}un_{n,\cong L}$ is not quasi-compact, a contradiction.
\end{proof}
Hereafter we assume $e \neq 0$. By Lemma \ref{1408-1} and Lemma \ref{lb-from} we have 
\[
(\mathscr{B}un_n^d)_{\mathcal{L}}^{a\text{-}ss} \subseteq (\mathscr{B}un_n^d)_{\mathcal{L}}^{\Theta\text{-}ss}=(\mathscr{B}un_n^d)_{\mathcal{L}_{\det}^{\otimes e}}^{\Theta\text{-}ss}.
\]
Further, we can compute that
\begin{align*}
(\mathscr{B}un_n^d)_{\mathcal{L}}^{a\text{-}ss} \subseteq (\mathscr{B}un_n^d)_{\mathcal{L}}^{\Theta\text{-}ss}=(\mathscr{B}un_n^d)_{\mathcal{L}_{\det}^{\otimes e}}^{\Theta\text{-}ss} &=\begin{cases}
(\mathscr{B}un_n^d)_{\mathcal{L}_{\det}}^{\Theta\text{-}ss} & \text{if } e>0 \\
(\mathscr{B}un_n^d)_{\mathcal{L}_{\det}^{-1}}^{\Theta\text{-}ss} & \text{if } e<0
\end{cases} \\
&=\begin{cases}
\mathscr{B}un_n^{d,ss} & \text{if } e>0 \text{ (by \cite[Lemma 1.11]{MR3758902})} \\
\emptyset & \text{if } e<0.
\end{cases}
\end{align*}
Indeed, any vector bundle $\mathcal{E}$ of rank $n \geq 2$ admits a sub-line bundle $\mathcal{L}$ with $\mu(\mathcal{L})<\mu(\mathcal{E})$. By \cite[1.E.b and 1.E.c]{MR3758902} the two step filtration $0 \subseteq \mathcal{L} \subseteq \mathcal{E}$ defines a testing morphism $f: [\mathbf{A}^1/\mathbf{G}_{m}] \to \mathscr{B}un_n^d$ with $f(1) \cong \mathcal{E}$ and it is computed that
\[
\mathrm{wt}_{\mathbf{G}_{m}}(f^*\mathcal{L}_{\det}^{-1})=-\mathrm{wt}_{\mathbf{G}_{m}}(f^*\mathcal{L}_{\det})=-2(n-1)(\mu(\mathcal{L})-\mu(\mathcal{E}/\mathcal{L}))>0.
\]
This shows that no vector bundle of rank $n \geq 2$ is $\Theta$-semistable with respect to $\mathcal{L}_{\det}^{-1}$.
\subsection{Applications}
Recall that the moduli stack of simple vector bundles admits a (in general non-separated) good moduli space (see \cite[Theorem 1]{MR0184252} or via rigidification by the central $\gm$ automorphisms as in \cite[Theorem 5.1.5]{MR2007376}). It is known to experts that this moduli space is not a scheme and such fact can be deduced directly from Theorem \ref{tba}.
\begin{cor}\label{1938-new}
If there exist simple unstable vector bundles of rank $n \geq 2$ and degree $d$ (e.g. if $g_C \geq 4$), then the good moduli space of simple vector bundles of rank $n$ and degree $d$ is only an algebraic space, not a scheme.
\end{cor}
\begin{rmk}\label{simple-exist}
Simple unstable vector bundles of rank $n \geq 2$ and degree $d$ are quite abundant. Here we sketch a construction generalizing \cite[Remark 12.3]{MR0184252}. Twisting by a degree $1$ line bundle we may assume that $0 < d \leq n$. Choose a stable vector bundle $\mathcal{V}$ of rank $n-1$ and degree $d$ such that $H^0(C,\mathcal{V})=0$ and $H^1(C,\mathcal{V}) \neq 0$ (such a vector bundle exists whenever $g_C \geq 2$ and $d<(n-1)(g_C-1)$). For any non-trivial extension $0 \to \mathcal{V} \to \mathcal{E} \to \mathcal{O}_C \to 0$ we claim that $\mathcal{E}$ is unstable and simple. Indeed, observe that any non-zero endomorphism $\psi$ of $\mathcal{E}$ induces an endomorphism of $\mathcal{V}$ since $\mathcal{V}$ is stable of positive slope. As $\mathcal{V}$ is stable this endomorphism is given by a scalar, say $\lambda$. Then the endomorphism $\psi-\lambda \cdot \mathrm{id}_{\mathcal{E}}$ of $\mathcal{E}$ yields a global section of $\mathcal{V}$, otherwise the defining sequence splits. By assumption $\mathcal{V}$ has no non-zero global sections. This shows that $\psi$ itself is a scalar and hence $\mathcal{E}$ is simple.
The conditions $0 < d \leq n \text{ (mod } n), g_C \geq 2$ and $d<(n-1)(g_C-1)$ are satisfied in the following cases: 
\begin{itemize}
\item 
$g_C \geq 4$.
\item
$g_C=3$ and $n \geq 3$, or $n=2$ and $d$ is odd.
\item 
$g_C=2$ and neither of $d, d+1$ is a multiple of $n$.
\end{itemize}
This construction can also be used to show that the good moduli space of simple vector bundles is non-separated, see Corollary \ref{simple-non-S}.
\end{rmk}
Finally, we remark that Theorem \ref{tba} provides another example that the local structure theorem in \cite[Theorem 1.1]{MR4088350} is optimal. To be precise we have the following:
\begin{cor}
The stack $\mathscr{B}un_G$ is only \'{e}tale-locally a quotient stack around unstable bundles with linearly reductive stabilizer groups, not Zariski-locally.
\end{cor}
\section{Existence criteria for vector bundles}\label{existence}
In this section, we translate the existence criteria for algebraic stacks to admit separated adequate moduli spaces stated in \cite{MR4665776} for vector bundles.
\begin{thm}[(\cite{MR4665776}, Theorem 5.4)]\label{thm0}
Let $\mathscr{X}$ be a quasi-compact algebraic stack, locally of finite type and with affine diagonal over a field $k$. Suppose $\mathscr{X}$ is locally reductive. Then $\mathscr{X}$ admits a separated adequate moduli space $X$ if and only if 
\begin{itemize}
\item 
$\mathscr{X}$ is $\Theta$-reductive and
\item
$\mathscr{X}$ is S-complete.
\end{itemize}
The separated adequate moduli space $X$ is proper if and only if $\mathscr{X}$ in addition satisfies the existence part of the valuative criterion for properness.
\end{thm}
\subsection{$\Theta$-reductivity}
Denote by 
\begin{itemize}
\item
$\mathrm{B}\mathbf{G}_m:=[\mathrm{Spec}(\mathbf{Z})/\mathbf{G}_m]$ the classifying stack of the multiplicative group $\mathbf{G}_m$.
\item
$\Theta:=[\mathbf{A}^1/\mathbf{G}_m]$ the quotient stack defined by the standard contracting action of $\mathbf{G}_m$ on the affine line $\mathbf{A}^1=\mathrm{Spec}(\mathbf{Z}[t])$.
\end{itemize}
Both stacks are defined over $\mathrm{Spec}(\mathbf{Z})$ and thus pull back to any base. For any DVR $R$ with fraction field $K$ and residue field $\kappa$, let $\Theta_R:=\Theta \times \mathrm{Spec}(R)$ and let $0:=[0/\mathbf{G}_m] \times \mathrm{Spec}(\kappa)$ be its unique closed point.
\begin{defn}[(\cite{MR4665776}, Definition 3.10)]
A locally noetherian algebraic stack $\mathscr{X}$ over a field $k$ is $\Theta$\emph{-reductive} if for every DVR $R$, any commutative diagram
\[
\begin{tikzcd}
\Theta_R-\{0\} \ar[r] \ar[d,hook] & \mathscr{X} \ar[d] \\
\Theta_R \ar[ur,dashed,"{\exists !}"'] \ar[r] & \mathrm{Spec}(k)
\end{tikzcd}
\]
of solid arrows can be uniquely filled in.
\end{defn}
\begin{rmk}\label{rmk:theta-red}
Observe that $\Theta_R-\{0\}$ is covered by the two open substacks
\[
[\mathbf{A}^1-\{0\}/\mathbf{G}_m] \times \mathrm{Spec}(R) \cong \mathrm{Spec}(R) \text{ and } [\mathbf{A}^1/\mathbf{G}_m] \times \mathrm{Spec}(K)=\Theta_K
\]
glued along $\mathrm{Spec}(K)$. Thus, a morphism $\Theta_R-\{0\} \to \mathscr{X}$ corresponds to the data of morphisms $\mathrm{Spec}(R) \to \mathscr{X}$ and $\Theta_K \to \mathscr{X}$ together with an isomorphism of their restrictions to $\mathrm{Spec}(K)$.
\end{rmk}
To investigate $\Theta$-reductivity for open substacks of $\mathscr{B}un_n^d$, we enlarge the target and classify morphisms $\Theta \to \mathscr{C}oh_n^d$ to the stack of coherent sheaves. By definition these morphisms correspond to coherent sheaves over $C \times \mathbf{A}^1$ flat over $\mathbf{A}^1$ and $\mathbf{G}_m$-equivariant for the action defined on the coordinate parameter, which can be further characterized using the Rees construction.
\begin{lem}[(\cite{MR3758902}, Lemma 1.10)]\label{1106}
Let $\mathscr{U} \subseteq \mathscr{C}oh_n^d$ be an open substack. Let $T$ be a locally noetherian scheme over $k$. A morphism $\Theta \times T \to \mathscr{U}$ is equivalent to the data of 
\begin{itemize}
\item 
a coherent sheaf $\mathcal{E}_T$ over $C \times T$, flat over $T$ such that for any point $t \in T$ the fiber $\mathcal{E}_t$ is coherent of rank $n$ and degree $d$.
\item 
a finite $\mathbf{Z}$-graded ascending filtration $\mathcal{E}_T^\bullet$ of $\mathcal{E}_T$ such that each $\mathcal{E}_T^i/\mathcal{E}_T^{i-1}$ is flat over $T$ and the associated graded sheaf $\mathrm{gr}(\mathcal{E}_T^\bullet) \in \mathscr{U}(T)$.
\end{itemize}
\end{lem}
\begin{cor}[(\cite{MR4480534}, Proposition 3.8)]
The stack $\mathscr{C}oh_n^d$ is $\Theta$-reductive.
\end{cor}
\begin{proof}
Given a DVR $R$ over $k$ with fraction field $K$, according to Remark \ref{rmk:theta-red} and Lemma \ref{1106} a morphism $\Theta_R-\{0\} \to \mathscr{C}oh_n^d$ is equivalent to a coherent sheaf $\mathcal{E}_R$ over $C \times \mathrm{Spec}(R)$, flat over $R$ and a finite $\mathbf{Z}$-graded ascending filtration $\mathcal{E}_K^\bullet$ of the generic fiber $\mathcal{E}_K$. Regard $\mathcal{E}_R$ as a subsheaf of $\mathcal{E}_K$ and define $\mathcal{E}_R^i:=\mathcal{E}_K^i \cap \mathcal{E}_R$, where the intersection is taken inside $\mathcal{E}_K$. Since $\mathcal{E}_R^i/\mathcal{E}_R^{i-1}$ is a subsheaf of $\mathcal{E}_K^i/\mathcal{E}_K^{i-1}$, it is torsion-free and hence flat over $R$. By Lemma \ref{1106} the finite $\mathbf{Z}$-graded ascending filtration $\mathcal{E}_R^\bullet$ of $\mathcal{E}_R$ defines a morphism $\Theta_R \to \mathscr{C}oh_n^d$ extending $\Theta_R-\{0\} \to \mathscr{C}oh_n^d$. The uniqueness of such an extension follows as $\mathcal{E}_R^i \subseteq \mathcal{E}_R$ is the unique subsheaf with generic fiber $\mathcal{E}_K^i$ such that $\mathcal{E}_R^i/\mathcal{E}_R^{i-1}$ is flat over $R$.
\end{proof}
The fact that $\mathscr{C}oh_n^d$ is $\Theta$-reductive simplifies the verification of $\Theta$-reductivity for its open substacks. If we complete the set-up of $\Theta$-reductivity for $\mathscr{U} \subseteq \mathscr{C}oh_n^d$ as
\[
\begin{tikzcd}
\Theta_R-\{0\} \ar[r] \ar[d,hook] \ar[dr] & \mathscr{U} \ar[d,hook] \\
\Theta_R \ar[r,dashed,"\exists !"'] & \mathscr{C}oh_n^d,
\end{tikzcd}
\]
then the dotted arrow can be uniquely filled in. To check whether $\mathscr{U}$ is $\Theta$-reductive, it suffices to check whether the image of $0 \in \Theta_R$ lies in $\mathscr{U}$. Together with Remark \ref{rmk:theta-red} and Lemma \ref{1106} we summarize this as follows.
\begin{prop}\label{thetaforopen1}
An open substack $\mathscr{U} \subseteq \mathscr{C}oh_n^d$ is $\Theta$-reductive if and only if for
\begin{itemize}
\item
every DVR $R$ over $k$ with fraction field $K$ and residue field $\kappa$,
\item
every family $\mathcal{E}_R \in \mathscr{U}(R)$, and
\item
every filtration $\mathcal{E}_K^\bullet$ of the generic fiber $\mathcal{E}_K$ with $\mathrm{gr}(\mathcal{E}_K^\bullet) \in \mathscr{U}(K)$,
\end{itemize}
the uniquely induced filtration $\mathcal{E}_\kappa^\bullet$ of the special fiber $\mathcal{E}_\kappa$ satisfies $\mathrm{gr}(\mathcal{E}_\kappa^\bullet) \in \mathscr{U}(\kappa)$.
\end{prop}
For vector bundles Proposition \ref{thetaforopen1} provides a first obstruction to $\Theta$-reductivity as the associated graded sheaf $\mathrm{gr}(\mathcal{E}_\kappa^\bullet)$ might have torsion. Indeed, if an open substack of vector bundles ``supports'' a family for which the generic fiber admits a subbundle that extends only to a subsheaf of the special fiber, then it cannot be $\Theta$-reductive. Such families are prototypes of the so-called $\Theta$-testing families. To start, we prove a technical lemma that given two families of coherent sheaves, any extension of their special fibers comes from a global one.
\begin{lem}[(Extension family)]\label{2321}
Let $R$ be a DVR over $k$ with fraction field $K$ and residue field $\kappa$. Given
\begin{itemize}
\item
two families $\mathcal{G}_{1,R} \in \mathscr{C}oh_{n_1}^{d_1}(R)$ and $\mathcal{G}_{2,R} \in \mathscr{C}oh_{n_2}^{d_2}(R)$, and
\item
an element $[\varrho] \in \mathrm{Ext}^1(\mathcal{G}_{2,\kappa},\mathcal{G}_{1,\kappa})$,
\end{itemize}
then there exists an extension family $\mathcal{G}_{R} \in \mathscr{C}oh_{n}^{d}(R)$ of $\mathcal{G}_{2,R}$ by $\mathcal{G}_{1,R}$ with special fiber $[\varrho]$, i.e.,
\begin{itemize}
\item
the generic fiber $\mathcal{G}_K$ fits into a short exact sequence 
\[
0 \to \mathcal{G}_{1,K} \to \mathcal{G}_K \to \mathcal{G}_{2,K} \to 0, \text{ and}
\]
\item
the special fiber $\mathcal{G}_\kappa$ fits into the short exact sequence corresponding to
\[
[\varrho]: 0 \to \mathcal{G}_{1,\kappa} \to \mathcal{G}_\kappa \to \mathcal{G}_{2,\kappa} \to 0.
\]
\end{itemize}
\end{lem}
\begin{proof}
The desired extension family is constructed via the $\mathscr{E}xt$-stack (see, e.g. \cite[Appendix A]{MR2605167}). Consider the stack $\mathscr{E}xt(t_2,t_1)$ parametrizing extensions
\[
0 \to \mathcal{E}_1 \to \mathcal{E} \to \mathcal{E}_2 \to 0,
\]
where $\mathcal{E}_i$ is a coherent sheaf over $C$ of type $t_i:=(n_i,d_i)$ for $i=1,2$. It comes with two natural morphisms. 
Define a morphism $\mathrm{pr}_2: \mathscr{E}xt(t_2,t_1) \to \mathscr{C}oh_n^d$ via  $[0 \to \mathcal{E}_1 \to \mathcal{E} \to \mathcal{E}_2 \to 0] \mapsto \mathcal{E}$. Furthermore, the assignment $[0 \to \mathcal{E}_1 \to \mathcal{E} \to \mathcal{E}_2 \to 0] \mapsto (\mathcal{E}_1,\mathcal{E}_2)$ defines a morphism
\[
\mathrm{pr}_{13}: \mathscr{E}xt(t_2,t_1) \to \mathscr{C}oh_{n_1}^{d_1} \times \mathscr{C}oh_{n_2}^{d_2}
\]
which is a generalized vector bundle (see \cite[Proposition A.3 (ii)]{MR2605167} or \cite[Corollary 3.2]{MR3293805}). This implies that in the following Cartesian diagram
\[
\begin{tikzcd}
\mathscr{E}xt(\mathcal{G}_{2,R},\mathcal{G}_{1,R}) \ar[rr] \ar[d] & \ & \mathscr{E}xt(t_2,t_1) \ar[d,"{\text{pr}_{13}}"] \\
\mathrm{Spec}(R) \ar[u,bend left,dashed,"{[s]}"] \ar[rr,"{(\mathcal{G}_{1,R},\mathcal{G}_{2,R})}"'] & & \mathscr{C}oh_{n_1}^{d_1} \times \mathscr{C}oh_{n_2}^{d_2} \ar[ul,phantom,"\lrcorner"]
\end{tikzcd}
\]
there exists a 2-term complex $[V_0 \to V_1]$ of vector bundles over $\mathrm{Spec}(R)$ such that 
\[
\mathscr{E}xt(\mathcal{G}_{2,R},\mathcal{G}_{1,R}) \cong [V_1/V_0].
\]
If there is a section $[s]: \mathrm{Spec}(R) \to \mathscr{E}xt(\mathcal{G}_{2,R},\mathcal{G}_{1,R})$ mapping $\mathrm{Spec}(\kappa)$ to $[\varrho]$, then the composition
\[
\mathrm{Spec}(R) \xrightarrow{[s]} \mathscr{E}xt(\mathcal{G}_{2,R},\mathcal{G}_{1,R}) \to \mathscr{E}xt(t_2,t_1) \xrightarrow{\mathrm{pr}_2} \mathscr{C}oh_n^d
\]
gives rise to the desired family $\mathcal{G}_{R} \in \mathscr{C}oh_n^d(R)$. To show the existence of such section, we lift $[\varrho]$ to $\varrho \in V_{1,\kappa}$, i.e.,
\[
\begin{tikzcd}
\varrho \in V_{1,\kappa} \ar[r,hook] \ar[d,twoheadrightarrow] & V_1 \ar[d,twoheadrightarrow,"\sigma"] \\
\left[\varrho\right] \in \mathscr{E}xt(\mathcal{G}_{2,\kappa},\mathcal{G}_{1,\kappa}) \ar[r,hook] \ar[d] & \mathscr{E}xt(\mathcal{G}_{2,R},\mathcal{G}_{1,R}) \cong [V_1/V_0] \ar[d] \\
\mathrm{Spec}(\kappa) \ar[r,hook] & \mathrm{Spec}(R).
\end{tikzcd}
\]
Since $\mathrm{Spec}(R)$ is affine, any element of the special fiber of a quasi-coherent sheaf lifts to a global section over $\mathrm{Spec}(R)$. Thus, there is a section $s: \mathrm{Spec}(R) \to V_1$ mapping $\mathrm{Spec}(\kappa) \mapsto \varrho$. Then we take $[s]:=\sigma \circ s: \mathrm{Spec}(R) \to \mathscr{E}xt(\mathcal{G}_{2,R},\mathcal{G}_{1,R})$.
\end{proof}
Now we can state the construction of $\Theta$-testing families.
\begin{prop}[($\Theta$-testing families)]\label{0305}
Let $\mathscr{U} \subseteq \mathscr{B}un_n^d$ be an open substack. Let $R$ be a DVR over $k$ with fraction field $K$ and residue field $\kappa$. If there exists a triple
\[
(\mathcal{G}_{1,R},\mathcal{G}_{2,R},[\varrho]) \in \mathscr{C}oh_{n_1}^{d_1}(R) \times \mathscr{C}oh_{n_2}^{d_2}(R) \times \mathrm{Ext}^1(\mathcal{G}_{2,\kappa},\mathcal{G}_{1,\kappa})
\]
such that $\mathcal{G}_{1,K} \oplus \mathcal{G}_{2,K} \in \mathscr{U}(K)$, $\mathcal{G}_{1,\kappa} \oplus \mathcal{G}_{2,\kappa} \notin \mathscr{U}(\kappa)$, and $\mathcal{E}([\varrho]) \in \mathscr{U}(\kappa)$, where $\mathcal{E}([\varrho])$ is the extension vector bundle associated to $[\varrho]$, then $\mathscr{U}$ is not $\Theta$-reductive.
\end{prop}
\begin{proof}
Choose an extension family $\mathcal{G}_R \in \mathscr{C}oh_n^d(R)$ of $\mathcal{G}_{2,R}$ by $\mathcal{G}_{1,R}$ with special fiber $[\varrho]$ using Lemma \ref{2321} (and any such extension family will be called a $\Theta$-testing family). Since $\mathscr{U} \subseteq \mathscr{B}un_n^d$ is open, $\mathcal{G}_\kappa=\mathcal{E}([\varrho]) \in \mathscr{U}(\kappa)$ implies that $\mathcal{G}_R \in \mathscr{U}(R)$. By construction the filtration $\mathcal{G}_K^\bullet$ of $\mathcal{G}_K$ satisfies $\mathrm{gr}(\mathcal{G}_K^\bullet)=\mathcal{G}_{1,K} \oplus \mathcal{G}_{2,K} \in \mathscr{U}(K)$, while the uniquely induced filtration $\mathcal{G}_\kappa^\bullet$ of $\mathcal{G}_\kappa$ does not. Thus by Proposition \ref{thetaforopen1} the stack $\mathscr{U}$ is not $\Theta$-reductive.
\end{proof}
As a special case of $\Theta$-testing families, we let the first family be a constant family and the second family be given by an elementary modification of vector bundles (see, e.g. \cite[Definition 5.2.1]{MR2665168}). In this case, Proposition \ref{0305} takes the following simple form.
\begin{cor}\label{0305-1}
Let $\mathscr{U} \subseteq \mathscr{B}un_n^d$ be an open substack. If there exists a quadruple
\[
(\mathcal{E}_1,\mathcal{E}_2,D,[\varrho]) \in \mathscr{B}un_{n_1}^{d_1}(k) \times \mathscr{B}un_{n_2}^{d_2}(k) \times \mathrm{Div}^{\mathrm{eff}}(C) \times \mathrm{Ext}^1(\mathcal{E}^D_2 \oplus \mathcal{O}_D,\mathcal{E}_1),
\]
where $\mathcal{E}^D_2 \subseteq \mathcal{E}_2$ is an elementary modification of $\mathcal{E}_2$ along the effective divisor $D$, such that $\mathcal{E}_1 \oplus \mathcal{E}_2 \in \mathscr{U}(k)$ and $\mathcal{E}([\varrho]) \in \mathscr{U}(k)$, where $\mathcal{E}([\varrho])$ is the extension vector bundle associated to $[\varrho]$, then $\mathscr{U}$ is not $\Theta$-reductive.
\end{cor}
\begin{proof}
Let $R$ be a DVR over $k$. Let $\mathcal{G}_{1,R} \in \mathscr{C}oh_{n_1}^{d_1}(R)$ be the constant family given by $\mathcal{E}_1$. There is a surjective morphism $\mathcal{E}_2 \to \mathcal{O}_D$ that induces a short exact sequence $0 \to \mathcal{E}_2^D \to \mathcal{E}_2 \to \mathcal{O}_D \to 0$. Let $\mathcal{G}_{2,R} \in \mathscr{C}oh_{n_2}^{d_2}(R)$ be the degeneration family
given by the Rees construction applied to the filtration $0 \subseteq \mathcal{E}_2^D \subseteq \mathcal{E}_2$. Then the triple $(\mathcal{G}_{1,R},\mathcal{G}_{2,R},[\varrho])$ satisfies the condition of Proposition \ref{0305} and hence $\mathscr{U}$ is not $\Theta$-reductive.
\end{proof}
\subsection{S-completeness}
For any DVR $R$ with fraction field $K$, residue field $\kappa$, and uniformizer $\pi \in R$, as in \cite[\S 2.B]{MR3758902}, the ``separatedness test space'' is defined as the quotient stack
\[
\overline{\mathrm{ST}}_R:=[\mathrm{Spec}(R[x,y]/xy-\pi)/\mathbf{G}_m],
\]
where $x,y$ have $\mathbf{G}_m$-weights $1,-1$ respectively. Denote by $0:=\mathrm{B}\mathbf{G}_{m,\kappa}=[\mathrm{Spec}(\kappa)/\mathbf{G}_m]$ its unique closed point defined by the vanishing of both $x$ and $y$. 

\begin{defn}[(\cite{MR4665776}, Definition 3.38)]
A locally noetherian algebraic stack $\mathscr{X}$ over a field $k$ is \emph{S-complete} if, for every DVR $R$, any commutative diagram
\[
\begin{tikzcd}
\overline{\mathrm{ST}}_R -\{0\} \ar[r] \ar[d,hook] & \mathscr{X} \ar[d] \\
\overline{\mathrm{ST}}_R \ar[r] \ar[ur,dashed,"\exists !"'] & \mathrm{Spec}(k)
\end{tikzcd}
\]
of solid arrows can be uniquely filled in.
\end{defn}
\begin{rmk}\label{rmk:S-compl}
Observe that $\overline{\mathrm{ST}}_R-\{0\}$ is covered by the two open substacks
\[
(\overline{\mathrm{ST}}_R-\{0\})|_{x \neq 0} \cong \mathrm{Spec}(R) \text{ and } (\overline{\mathrm{ST}}_R-\{0\})|_{y \neq 0} \cong \mathrm{Spec}(R)
\]
glued along $\mathrm{Spec}(K)$, i.e.,
\[
\overline{\mathrm{ST}}_R-\{0\} \cong \mathrm{Spec}(R) \bigcup_{\mathrm{Spec}(K)} \mathrm{Spec}(R).
\]
Therefore, a morphism $\overline{\mathrm{ST}}_R-\{0\} \to \mathscr{X}$ is the data of two morphisms $\mathrm{Spec}(R) \to \mathscr{X}$ together with an isomorphism of their restrictions to $\mathrm{Spec}(K)$.
\end{rmk}
To investigate S-completeness for open substacks of $\mathscr{B}un_n^d$, as before we enlarge the target and classify morphisms $\overline{\mathrm{ST}}_R \to \mathscr{C}oh_n^d$. By definition these morphisms correspond to coherent sheaves over $C \times \overline{\mathrm{ST}}_R$ flat over $\overline{\mathrm{ST}}_R$, which can be further characterized as follows.
\begin{lem}[(\cite{MR4665776}, Corollary 7.14)]\label{0203}
Let $R$ be a DVR over $k$ with uniformizer $\pi$. Then a quasi-coherent sheaf $\mathcal{G}$ over $C \times \overline{\mathrm{ST}}_R$ corresponds to a $\mathbf{Z}$-graded coherent sheaf $\bigoplus_{i \in \mathbf{Z}} \mathcal{G}_i$ over $C \times \mathrm{Spec}(R)$ together with a diagram
\[
\begin{tikzcd}
\cdots \ar[r,bend right,"x"'] & \mathcal{G}_{i-1} \ar[r,bend right,"x"'] \ar[l,bend right,"y"'] & \mathcal{G}_i \ar[r,bend right,"x"'] \ar[l,bend right,"y"'] & \mathcal{G}_{i+1} \ar[r,bend right,"x"'] \ar[l,bend right,"y"'] & \cdots \ar[l,bend right,"y"']
\end{tikzcd}
\]
such that $xy=yx=\pi$. Moreover,
\begin{enumerate}
\item
$\mathcal{G}$ is coherent if each $\mathcal{G}_i$ is coherent, $x: \mathcal{G}_{i-1} \to \mathcal{G}_i$ is an isomorphism for $i \gg 0$, and $y: \mathcal{G}_i \to \mathcal{G}_{i-1}$ is an isomorphism for $i \ll 0$.
\item
$\mathcal{G}$ is flat over $\overline{\mathrm{ST}}_R$ if and only if the maps $x$, $y$, and the induced map $x: \mathcal{G}_{i-1}/y\mathcal{G}_i \to \mathcal{G}_i/y\mathcal{G}_{i+1}$ are injective.
\end{enumerate} 
\end{lem}

\begin{lem}\label{coh-s-co}
The stack $\mathscr{C}oh_n^d$ is S-complete.
\end{lem}
\begin{proof}
For every DVR $R$ over $k$ and any commutative diagram
\[
\begin{tikzcd}
\overline{\mathrm{ST}}_R-\{0\} \ar[r] \ar[d,hook,"j"'] & \mathscr{C}oh_n^d \ar[d] \\
\overline{\mathrm{ST}}_R \ar[r] \ar[ur,dashed,"\exists !"'] & \mathrm{Spec}(k)
\end{tikzcd}
\]
of solid arrows, we show there exists a unique dotted arrow filling it in.

\textsc{Uniqueness}. Let $\mathcal{E}$ be the $(\overline{\mathrm{ST}}_R-\{0\})$-flat coherent sheaf over $C \times (\overline{\mathrm{ST}}_R-\{0\})$ defined by the morphism $\overline{\mathrm{ST}}_R-\{0\} \to \mathscr{C}oh_n^d$. If a $\overline{\mathrm{ST}}_R$-flat coherent extension exists, then it must be isomorphic to the quasi-coherent sheaf $(\mathrm{id}_C \times j)_*\mathcal{E}$ over $C \times \overline{\mathrm{ST}}_R$ via the adjunction map, as the closed point $0 \in \overline{\mathrm{ST}}_R$ has codimension two and $C \times \overline{\mathrm{ST}}_R$ is normal.

\textsc{Existence}. Let $K$ be the fraction field of $R$ and $\pi \in R$ be a uniformizer. The morphism $\overline{\mathrm{ST}}_R-\{0\} \to \mathscr{C}oh_n^d$ is the same as two families $\mathcal{E}_R,\mathcal{E}'_R \in \mathscr{C}oh_n^d(R)$ together with an isomorphism $\mathcal{E}_K \cong \mathcal{E}'_K$ by Remark \ref{rmk:S-compl}. Let $a \in \mathbf{N}$ (resp., $b \in \mathbf{N}$) be the minimal integer such that $\mathcal{E}_R \subseteq \pi^{-a}\mathcal{E}'_R$ (resp., $\mathcal{E}'_R \subseteq \pi^{-b}\mathcal{E}_R$). Then we can define a $\mathbf{Z}$-graded coherent sheaf $\oplus_{i \in \mathbf{Z}} \mathcal{G}_i$ over $C \times \mathrm{Spec}(R)$ as follows:
\[
\mathcal{G}_i:=\begin{cases}
\mathcal{E}'_R & \text{if } i \leq -b \\
\pi^i\mathcal{E}_R \cap \mathcal{E}'_R & \text{if } -b<i<0 \\
\mathcal{E}_R \cap \pi^{-i}\mathcal{E}'_R & \text{if } 0 \leq i<a \\
\mathcal{E}_R & \text{if } i \geq a.
\end{cases}
\]
Let $x_i: \mathcal{G}_i \to \mathcal{G}_{i+1}$ and $y_i: \mathcal{G}_{i+1} \to \mathcal{G}_i$ act via
\[
\begin{tikzcd}
\cdots \ar[r,"\pi"',bend right] & \mathcal{E}'_R \ar[r,"\pi"',bend right] \ar[l,"\text{id}"',bend right] & \mathcal{G}_{-(b-1)} \ar[r,"\pi"',bend right] \ar[l,bend right,"\text{icl}"'] & \cdots \ar[r,"\pi"',bend right] \ar[l,bend right,"\text{icl}"'] & \mathcal{E}_R \cap \mathcal{E}'_R \ar[r,bend right,"\text{icl}"'] \ar[l,bend right,"\text{icl}"'] & \cdots \ar[l,"\pi"',bend right] \ar[r,bend right,"\text{icl}"'] & \mathcal{G}_{a-1} \ar[l,"\pi"',bend right] \ar[r,bend right,"\text{icl}"'] & \mathcal{E}_R \ar[l,"\pi"',bend right] \ar[r,"\text{id}"',bend right] & \cdots \ar[l,"\pi"',bend right] \\
\cdots & -b & -(b-1) & \cdots & 0 & \cdots & a-1 & a & \cdots
\end{tikzcd}
\]
By Lemma \ref{0203} this diagram defines a coherent sheaf over $C \times \overline{\mathrm{ST}}_R$ flat over $\overline{\mathrm{ST}}_R$, i.e., a morphism $\overline{\mathrm{ST}}_R \to \mathscr{C}oh_n^d$ which extends $\overline{\mathrm{ST}}_R-\{0\} \to \mathscr{C}oh_n^d$ by construction. Indeed, it remains to check that the induced map $x_i: \mathcal{G}_i/y_i\mathcal{G}_{i+1} \to \mathcal{G}_{i+1}/y_{i+1}\mathcal{G}_{i+2}$ is injective for all $i \in \mathbf{Z}$, i.e.,
\begin{gather*}
\cdots \to 0 \to \dfrac{\mathcal{E}'_R}{\pi^{-(b-1)}\mathcal{E}_R \cap \mathcal{E}'_R} \xrightarrow{\pi} \dfrac{\pi^{-(b-1)}\mathcal{E}_R \cap \mathcal{E}'_R}{\pi^{-(b-2)}\mathcal{E}_R \cap \mathcal{E}'_R} \xrightarrow{\pi} \cdots \xrightarrow{\pi} \dfrac{\pi^{-1}\mathcal{E}_R \cap \mathcal{E}'_R}{\mathcal{E}_R \cap \mathcal{E}'_R} \xrightarrow{\pi} \dfrac{\mathcal{E}_R \cap \mathcal{E}'_R}{\pi\mathcal{E}_R \cap\mathcal{E}'_R} \\ 
\xrightarrow{\mathrm{icl}} \dfrac{\mathcal{E}_R \cap \pi^{-1}\mathcal{E}'_R}{\pi\mathcal{E}_R \cap \pi^{-1}\mathcal{E}'_R} \xrightarrow{\mathrm{icl}} \cdots \xrightarrow{\mathrm{icl}} \dfrac{\mathcal{E}_R \cap \pi^{-(a-1)}\mathcal{E}'_R}{\pi\mathcal{E}_R} \xrightarrow{\mathrm{icl}} \mathcal{E}_k \xrightarrow{\mathrm{id}} \cdots
\end{gather*}
is injective in each degree. If $i<0$, then the induced map $x_i: \mathcal{G}_i/y_i\mathcal{G}_{i+1} \to \mathcal{G}_{i+1}/y_{i+1}\mathcal{G}_{i+2}$ is injective since in the following commutative diagram the top square is Cartesian
\[
\begin{tikzcd}
\mathcal{G}_{i+1}=\pi^{-(i+1)}\mathcal{E}_R \cap \mathcal{E}'_R \ar[rr,hook,"x_{i+1}=\pi"] \ar[d,hook,"{y_{i}=\text{icl}}"'] & \ & \pi^{-(i+2)}\mathcal{E}_R \cap \mathcal{E}'_R=\mathcal{G}_{i+2} \ar[d,hook,"{y_{i+1}=\text{icl}}"] \\
\mathcal{G}_i=\pi^{-i}\mathcal{E}_R \cap \mathcal{E}'_R \ar[rr,hook,"x_{i}=\pi"'] \ar[d,twoheadrightarrow] & & \pi^{-(i+1)}\mathcal{E}_R \cap \mathcal{E}'_R=\mathcal{G}_{i+1} \ar[d,twoheadrightarrow] \ar[ul,phantom,"\lrcorner"]\\
\dfrac{\mathcal{G}_i}{\mathcal{G}_{i+1}}=\dfrac{\pi^{-i}\mathcal{E}_R \cap \mathcal{E}'_R}{\pi^{-(i+1)}\mathcal{E}_R \cap \mathcal{E}'_R} \ar[rr,dashed,hook,"x_i=\pi"'] & & \dfrac{\pi^{-(i+1)}\mathcal{E}_R \cap \mathcal{E}'_R}{\pi^{-(i+2)}\mathcal{E}_R \cap \mathcal{E}'_R}=\dfrac{\mathcal{G}_{i+1}}{\mathcal{G}_{i+2}}.
\end{tikzcd}
\]
A similar argument works for the case $i>0$.
\end{proof}
The fact that $\mathscr{C}oh_n^d$ is S-complete simplifies the verification of S-completeness for its open substacks. If we complete the set-up of S-completeness for $\mathscr{U} \subseteq \mathscr{C}oh_n^d$ as
\[
\begin{tikzcd}
\overline{\mathrm{ST}}_R-\{0\} \ar[r] \ar[d,hook] \ar[dr] & \mathscr{U} \ar[d,hook] \\
\overline{\mathrm{ST}}_R \ar[r,dashed,"\exists !"'] & \mathscr{C}oh_n^d,
\end{tikzcd}
\]
then the dotted arrow can be uniquely filled in. To check whether $\mathscr{U}$ is S-complete, it suffices to check whether the image of $0 \in \overline{\mathrm{ST}}_R$ lies in $\mathscr{U}$.
\begin{rmk}\label{3de}
The $\Theta$-reductivity and S-completeness of the stack $\mathscr{C}oh(X)$ of coherent sheaves over a proper scheme $X$ over $k$ can be shown directly and uniformly as follows. Let $R$ be a DVR and let $\mathscr{X}$ be $\Theta_R$ or $\overline{\mathrm{ST}}_R$. The unique closed point $0 \in \mathscr{X}$ has codimension $2$ and we denote by $\mathscr{U}:=\mathscr{X}-\{0\}$ its open complement and $j: X \times \mathscr{U} \hookrightarrow X \times \mathscr{X}$ the open immersion. For any $\mathscr{U}$-flat coherent sheaf $\mathcal{F}$ over $X \times \mathscr{U}$, we claim that $j_*\mathcal{F}$ is coherent and $\mathscr{X}$-flat. For coherence, note that $\mathcal{F}$ is $\mathscr{U}$-flat, all of its associated points lie in the generic fiber of the projection $X \times \mathscr{U} \to \mathscr{U}$ by \cite[\href{https://stacks.math.columbia.edu/tag/05DB}{Tag 05DB}]{stacks-project}, so we can apply \cite[\href{https://stacks.math.columbia.edu/tag/0AWA}{Tag 0AWA}]{stacks-project} to conclude.
For $\mathscr{X}$-flatness, we just need to check it around $0$. As in the proof of \cite[Lemma 7.15]{MR4665776}, by the local criterion for flatness it suffices to show that the local Koszul complex
\[
0 \to j_*\mathcal{F} \xrightarrow{x \oplus (-y)} j_*\mathcal{F} \oplus j_*\mathcal{F} \xrightarrow{y \oplus x} j_*\mathcal{F}
\]
is exact. Here $x,y$ will be replaced by the regular sequence $\pi,t$ at $0$ if $\mathscr{X}=\Theta_R$.
By $\mathscr{U}$-flatness the sequence over $X \times \mathscr{U}$
\[
0 \to \mathcal{F} \xrightarrow{x \oplus (-y)} \mathcal{F} \oplus \mathcal{F} \xrightarrow{y \oplus x} \mathcal{F} \to 0
\]
is exact. Then we conclude by taking $j_*$ which is left exact.

Moreover, if $\mathcal{E}$ is another $\mathscr{X}$-flat coherent extension of $\mathcal{F}$, then the adjunction map $\mathcal{E} \to j_*\mathcal{F}$ is an isomorphism by \cite[\href{https://stacks.math.columbia.edu/tag/0AVM}{Tag 0AVM}]{stacks-project}. This proves the uniqueness.
\end{rmk}
By deformation theory any morphism $\overline{\mathrm{ST}}_R \to \mathscr{C}oh_n^d$ is determined by its restriction to the coordinate cross $xy=0$ as the target is smooth. Since $\overline{\mathrm{ST}}_R|_{xy=0}=\Theta_\kappa \cup_{\mathrm{B}\mathbf{G}_{m,\kappa}} \Theta_\kappa$ (where $\kappa$ is the residue field of $R$) and a morphism $\Theta \cup_{\mathrm{B}\mathbf{G}_m} \Theta \to \mathscr{C}oh_n^d$ corresponds to two ``oppositely'' filtered coherent sheaves by Lemma \ref{1106}, we first make this into a definition.
\begin{defn}
Two points $\mathcal{E} \ncong \mathcal{E}' \in \mathscr{C}oh_n^d(\mathbf{K})$ for some field $\mathbf{K}/k$ are said to have \emph{opposite filtrations} if there exist two finite $\mathbf{Z}$-graded filtrations of subsheaves
\begin{gather*}
    \mathcal{E}^\bullet: 0 \subseteq \cdots \subseteq \mathcal{E}^{i-1} \subseteq \mathcal{E}^{i} \subseteq \mathcal{E}^{i+1} \subseteq \cdots \subseteq \mathcal{E} \\ 
    \mathcal{E}'_\bullet: \mathcal{E}' \supseteq \cdots \supseteq \mathcal{E}'_{i-1} \supseteq \mathcal{E}'_{i} \supseteq \mathcal{E}'_{i+1} \supseteq \cdots \supseteq 0
\end{gather*}
such that $\mathcal{E}^i/\mathcal{E}^{i-1} \cong \mathcal{E}'_{i}/\mathcal{E}'_{i+1}$ for all $i \in \mathbf{Z}$.
\end{defn}
Here are some basic properties of coherent sheaves with opposite filtrations.
\begin{lem}\label{2235}
For any field $\mathbf{K}/k$ and any two points $\mathcal{E} \ncong \mathcal{E}' \in \mathscr{C}oh_n^d(\mathbf{K})$ with opposite filtrations $\mathcal{E}^\bullet, \mathcal{E}'_\bullet$, there exist
\begin{enumerate}
\item
isomorphisms $\mathrm{gr}(\mathcal{E}^\bullet) \cong \mathrm{gr}(\mathcal{E}'_\bullet)$ and $\det(\mathcal{E}) \cong \det(\mathcal{E}')$.
\item
non-zero morphisms $h: \mathcal{E} \to \mathcal{E}', h': \mathcal{E}' \to \mathcal{E}$ such that $h \circ h'=h' \circ h=0$.
\end{enumerate}
\end{lem}
\begin{proof}
Part (i) is clear from the definition. For part (ii) let $a \in \mathbf{Z}$ be the minimal integer such that $\mathcal{E}^a=\mathcal{E}$, which is also the minimal integer such that $\mathcal{E}'_{a+1}=0$. Let $b \in \mathbf{Z}$ be the maximal integer such that $\mathcal{E}'_b=\mathcal{E}'$, which is also the maximal integer such that $\mathcal{E}^{b-1}=0$. Then we take
\begin{gather*}
h: \mathcal{E} \twoheadrightarrow \mathcal{E}/\mathcal{E}^{a-1}=\mathcal{E}^a/\mathcal{E}^{a-1} \cong \mathcal{E}'_a/\mathcal{E}'_{a+1}=\mathcal{E}'_a \hookrightarrow \mathcal{E}' \\
h': \mathcal{E}' \twoheadrightarrow \mathcal{E}'/\mathcal{E}'_{b+1}=\mathcal{E}'_{b}/\mathcal{E}'_{b+1} \cong \mathcal{E}^{b}/\mathcal{E}^{b-1}=\mathcal{E}^{b} \hookrightarrow \mathcal{E}.
\end{gather*}
\end{proof}
Now we can prove a modular characterization of S-completeness for coherent sheaves in terms of points with opposite filtrations, where a powerful tool - coherent completeness - is employed.
\begin{prop}\label{prop:S-comple}
An open substack $\mathscr{U} \subseteq \mathscr{C}oh_n^d$ is S-complete if and only if for any field $\mathbf{K}/k$ and $\mathcal{E} \ncong \mathcal{E}' \in \mathscr{U}(\mathbf{K})$ with opposite filtrations $\mathcal{E}^\bullet,\mathcal{E}'_\bullet$, we have $\mathrm{gr}(\mathcal{E}^\bullet) \cong \mathrm{gr}(\mathcal{E}'_\bullet) \in \mathscr{U}(\mathbf{K})$.
\end{prop}
\begin{proof}
\textsc{If Part}: Let $R$ be a DVR over $k$ with residue field $\kappa$. For any commutative diagram
\[
\begin{tikzcd}
\overline{\mathrm{ST}}_R-\{0\} \ar[r] \ar[d,hook] & \mathscr{U} \ar[d,hook] \\
\overline{\mathrm{ST}}_R \ar[r,"{\exists ! f}"'] \ar[ur,dashed] & \mathscr{C}oh_n^d
\end{tikzcd}
\]
of solid arrows, where $f$ is the unique morphism given by Lemma \ref{coh-s-co}, we show there exists a dotted arrow filling it in, i.e., $f(0) \in \mathscr{U}$. Keep the notation in the proof of Lemma \ref{coh-s-co}. By \cite[Corollary 7.14]{MR4665776} the restriction $f|_{x=0}$ (resp., $f|_{y=0}$) gives rise to a filtration $\mathcal{E}_\kappa^{\bullet}$ (resp., $(\mathcal{E}'_\kappa)_{\bullet}$) of $\mathcal{E}_\kappa$ (resp., $\mathcal{E}'_\kappa$)
\begin{gather*}
\mathcal{E}_\kappa^{\bullet}: \cdots \to 0 \to \dfrac{\mathcal{E}'_R}{\pi^{-(b-1)}\mathcal{E}_R \cap \mathcal{E}'_R} \to \cdots \to \dfrac{\mathcal{E}_R \cap \pi^{-(a-1)}\mathcal{E}'_R}{\pi\mathcal{E}_R} \to \mathcal{E}_\kappa \to \cdots \\
\left(\text{resp., } \cdots \leftarrow \mathcal{E}'_\kappa \leftarrow \dfrac{\pi^{-(b-1)}\mathcal{E}_R \cap \mathcal{E}'_R}{\pi\mathcal{E}'_R} \leftarrow \cdots \leftarrow \dfrac{\mathcal{E}_R}{\mathcal{E}_R \cap \pi^{-(a-1)}\mathcal{E}'_R} \leftarrow 0 \leftarrow \cdots :(\mathcal{E}'_\kappa)_{\bullet} \right).
\end{gather*}
Furthermore, we have $f(0)=\mathrm{gr}(\mathcal{E}_\kappa^{\bullet}) \cong \mathrm{gr}((\mathcal{E}'_\kappa)_{\bullet})$. By construction there exists an isomorphism 
\[
\mathcal{E}_\kappa^{i}/\mathcal{E}_\kappa^{i-1} \cong (\mathcal{E}_\kappa')_{i}/(\mathcal{E}_\kappa')_{i+1} \text{ for all } i \in \mathbf{Z},
\]
i.e., the two filtrations $\mathcal{E}_\kappa^{\bullet}$ and $(\mathcal{E}'_\kappa)_{\bullet}$ are opposite. Then by assumption $\mathrm{gr}(\mathcal{E}_\kappa^{\bullet}) \cong \mathrm{gr}((\mathcal{E}'_\kappa)_{\bullet}) \in \mathscr{U}(\kappa)$, i.e., $f(0) \in \mathscr{U}(\kappa)$ and we are done.

\textsc{Only If Part}: Let $\mathcal{E},\mathcal{E}' \in \mathscr{U}(\mathbf{K})$ be two points (for some field $\mathbf{K}/k$) with opposite filtrations $\mathcal{E}^\bullet,\mathcal{E}'_\bullet$. We show $\mathrm{gr}(\mathcal{E}^\bullet) \cong \mathrm{gr}(\mathcal{E}'_\bullet) \in \mathscr{U}(\mathbf{K})$. For this let $R:=\mathbf{K}[[\pi]]$ and $A:=R[x,y]/xy-\pi$. Let $\mathbf{G}_m$ act on $A$ such that $x$ and $y$ have $\mathbf{G}_m$-weights $1$ and $-1$ respectively. The two opposite filtrations $\mathcal{E}^\bullet,\mathcal{E}'_\bullet$ define a family over the coordinate cross $\mathrm{Spec}(\mathbf{K}[x,y]/xy) \cong \mathrm{Spec}(A/\pi)$. Indeed, applying the Rees construction to $\mathcal{E}^\bullet$ (resp., $\mathcal{E}'_\bullet$) gives rise to a $\mathbf{G}_m$-invariant family over $\mathrm{Spec}(\mathbf{K}[x])=\mathrm{Spec}(A/(\pi,y))$ (resp., $\mathrm{Spec}(\mathbf{K}[y])=\mathrm{Spec}(A/(\pi,x))$). Using the isomorphisms $\mathcal{E}^i/\mathcal{E}^{i-1} \cong \mathcal{E}'_{i}/\mathcal{E}'_{i+1}$ to glue these two families 
at the origin $\mathrm{Spec}(\mathbf{K})=\mathrm{Spec}(A/(\pi,x,y))$, we obtain a $\mathbf{G}_m$-invariant family over 
\[
\mathrm{Spec}(\mathbf{K}[x]) \cup_{\mathrm{Spec}(\mathbf{K})} \mathrm{Spec}(\mathbf{K}[y]) \cong \mathrm{Spec}(\mathbf{K}[x,y]/xy) \cong \mathrm{Spec}(A/\pi).
\]
This defines a morphism $f_1: [\mathrm{Spec}(A/\pi)/\mathbf{G}_m] \to \mathscr{C}oh_n^d$. Then we want to find compatible lifts $f_i: [\mathrm{Spec}(A/\pi^i)/\mathbf{G}_m] \to \mathscr{C}oh_n^d$ of $f_1$ for all $i \geq 1$, i.e., a commutative diagram
\[
\begin{tikzcd}
\left[\mathrm{Spec}(A/\pi^{i-1})/\mathbf{G}_m\right] \ar[rr,"f_{i-1}"] \ar[d,hook] & & \mathscr{C}oh_n^d \ar[d] \\
\left[\mathrm{Spec}(A/\pi^i)/\mathbf{G}_m\right] \ar[urr,dashed,"f_i"'] \ar[rr] & & \mathrm{Spec}(k).
\end{tikzcd}
\]
Since $\mathscr{C}oh_n^d$ is smooth over $k$, any morphism $\mathrm{Spec}(A/\pi) \to \mathscr{C}oh_n^d$ can be lifted infinitesimally to $\mathrm{Spec}(A/\pi^i)$ for all $i \geq 1$. The obstruction to the existence of a $\mathbf{G}_m$-equivariant lifting lies in $H^1\left(\mathrm{B}\mathbf{G}_m,f_{i-1}^*\mathbf{T}_{\mathscr{C}oh_n^d/k}|_{\mathrm{B}\mathbf{G}_m} \otimes (\pi^{i-1})/(\pi^i)\right)=0$. Thus there exists a compatible family of morphisms $[\mathrm{Spec}(A/\pi^i)/\mathbf{G}_m] \to \mathscr{C}oh_n^d$. Finally by coherent completeness (see \cite[Corollary 1.5]{MR4009673}) this yields a morphism 
\[
f: \overline{\mathrm{ST}}_R=[\mathrm{Spec}(A)/\mathbf{G}_m] \to \mathscr{C}oh_n^d.
\]
such that $f(0)=\mathrm{gr}(\mathcal{E}^\bullet) \cong \mathrm{gr}(\mathcal{E}'_\bullet)$ and $\overline{\mathrm{ST}}_R-\{0\}$ maps to $\mathscr{U} \subseteq \mathscr{C}oh_n^d$. Since $\mathscr{U}$ is S-complete, $f(0) \in \mathscr{U}$, i.e., $\mathrm{gr}(\mathcal{E}^\bullet) \cong \mathrm{gr}(\mathcal{E}'_\bullet) \in \mathscr{U}(\mathbf{K})$.
\end{proof}
To verify S-completeness using Proposition \ref{prop:S-comple}, we need to know when two coherent sheaves have opposite filtrations. In some situations this can be fully understood, e.g. the stack $\mathscr{B}un_n^{d,s}$ of stable vector bundles has no points with opposite filtrations by Lemma \ref{2235} (ii) and thus is S-complete. To illustrate this characterization a little bit more we present one example and one non-example. First we give a quick proof that $\mathscr{B}un_n^{d,ss}$ is S-complete. For a different proof in a more general setup, see \cite[Lemma 8.4]{MR4665776}.
\begin{cor}
The stack $\mathscr{B}un_n^{d,ss}$ of semistable vector bundles is S-complete, i.e., the adequate moduli space of semistable vector bundles of rank $n$ and degree $d$ is separated.
\end{cor}
\begin{proof}
By Proposition \ref{prop:S-comple}, it suffices to prove that if $\mathcal{E} \ncong \mathcal{E}' \in \mathscr{B}un_n^{d,ss}(\mathbf{K})$ for some field $\mathbf{K}/k$ have opposite filtrations $\mathcal{E}^\bullet,\mathcal{E}'_\bullet$, 
then any $\mathcal{E}^{i}/\mathcal{E}^{i-1} \neq 0$ is semistable of slope $\mu:=d/n$, 
as this will imply that $\mathrm{gr}(\mathcal{E}^\bullet)=\mathrm{gr}(\mathcal{E}'_\bullet) \in \mathscr{B}un_n^{d,ss}(\mathbf{K})$. 
A simple computation about slope yields that it is equivalent to show any $\mathcal{E}^{i} \neq 0$ has slope $\mu$ (and hence is automatically semistable). Let $b \in \mathbf{Z}$ be the maximal integer such that $\mathcal{E}^{b-1}=0$. We prove by induction that
\[
\mu(\mathcal{E}^i)=\mu=\mu(\mathcal{E}'_{i+1}) \text{ for } i \geq b.
\]
Consider the morphism $\mathcal{E}' \twoheadrightarrow \mathcal{E}'/\mathcal{E}'_{b+1}=\mathcal{E}'_{b}/\mathcal{E}'_{b+1} \cong \mathcal{E}^{b}/\mathcal{E}^{b-1}=\mathcal{E}^{b} \hookrightarrow \mathcal{E}$. The semistability of $\mathcal{E}$ and $\mathcal{E}'$ yields that
\[
\mu=\mu(\mathcal{E}') \leq \mu(\mathcal{E}'/\mathcal{E}'_{b+1})=\mu(\mathcal{E}^b) \leq \mu(\mathcal{E})=\mu
\]
and hence $\mu(\mathcal{E}^b)=\mu=\mu(\mathcal{E}'_{b+1})$. Suppose $\mu(\mathcal{E}^{i-1})=\mu=\mu(\mathcal{E}'_{i})$, then $\mu(\mathcal{E}^{i})=\mu=\mu(\mathcal{E}'_{i+1})$ follows from the inequality
\[
\mu=\mu(\mathcal{E}^{i-1}) \geq \mu(\mathcal{E}^i) \geq \mu(\mathcal{E}^i/\mathcal{E}^{i-1})=\mu(\mathcal{E}'_{i}/\mathcal{E}'_{i+1}) \geq \mu(\mathcal{E}'_{i})=\mu.
\]
The separatedness of the adequate moduli space follows from \cite[Proposition 3.48 (2)]{MR4665776}.
\end{proof}
Next, we show that the good moduli space of simple vector bundles is in general non-separated. This is certainly well-known (e.g. in the coprime case it follows from a straightforward generalization of \cite[Remark 12.3]{MR0184252}) and we can give a direct proof here.
\begin{cor}\label{simple-non-S}
Let $n \geq 2$ and assume that the triple $(g_C,n,d)$ satisfies one of the conditions in Remark \ref{simple-exist} (e.g. if $g_C \geq 4$). 
Then the stack $\mathscr{B}un_n^{d,simple}$ of simple vector bundles is not S-complete, i.e., the good moduli space of simple vector bundles of rank $n$ and degree $d$ is non-separated.
\end{cor}
\begin{proof}
Twisting by a degree 1 line bundle we may assume that $0< d \leq n$. By Proposition \ref{prop:S-comple} it suffices to find two simple vector bundles with opposite filtrations. Let $\mathcal{V}$ be a stable vector bundle of rank $n-1$ and degree $d$ such that 
\[
H^0(C,\mathcal{V})=0 \text{ and } H^1(C,\mathcal{V}) \neq 0.
\]
Such a vector bundle exists by our assumption on $(g_C,n,d)$. Then $H^0(C,\mathcal{V}^{-1})=0$ as $\mathcal{V}^{-1}$ is stable of negative slope and $H^1(C,\mathcal{V}^{-1}) \neq 0$ by Riemann-Roch. A similar argument to the one given in Remark \ref{simple-exist} yields that any non-split extensions
\[
0 \to \mathcal{V} \to \mathcal{E} \to \mathcal{O}_C \to 0 \text{ and } 0 \to \mathcal{O}_C \to \mathcal{E}' \to \mathcal{V} \to 0
\]
are simple. Furthermore, they have opposite filtrations by construction and we conclude. The non-separatedness of the good moduli space follows from \cite[Proposition 3.48 (2)]{MR4665776}.
\end{proof}
There is a sufficient condition for vector bundles to have opposite filtrations, which can be seen as a partial inverse to Lemma \ref{2235}.
\begin{prop}\label{173e}
Given two points $\mathcal{E} \ncong \mathcal{E}'\in\mathscr{B}un_n^d(\mathbf{K})$ for some field $\mathbf{K}/k$. If there exist
\begin{enumerate}
\item
an isomorphism $\lambda: \det(\mathcal{E}) \xrightarrow{\sim} \det(\mathcal{E}')$, and
\item
morphisms $h: \mathcal{E} \to \mathcal{E}'$ of rank $r$ and $h': \mathcal{E}' \to \mathcal{E}$ of rank $n-r$ such that $h \circ h'=h' \circ h=0$ and the following diagram 
\begin{equation}\label{comm-up}
\begin{tikzcd}
\wedge^r(\mathcal{E}) \ar[rr,"{\wedge^r(h)}"] \ar[d,"{\varphi_\mathcal{E}}"',"\cong"] & & \wedge^r(\mathcal{E}') \ar[d,"{\varphi_{\mathcal{E}'}}","\cong"'] \\
\wedge^{n-r}(\mathcal{E})^{-1} \otimes \det(\mathcal{E}) \ar[rr,"{\wedge^{n-r}(h')^{-1} \otimes \lambda}"'] & & \wedge^{n-r}(\mathcal{E}')^{-1} \otimes \det(\mathcal{E}')
\end{tikzcd}
\end{equation}
commutes, where $\varphi_{\mathcal{E}}$ and $\varphi_{\mathcal{E}'}$ are the canonical isomorphisms,
\end{enumerate}
then $\mathcal{E}$ and $\mathcal{E}'$ have opposite filtrations.
\end{prop}
\begin{proof}
Suppose there exist morphisms $h: \mathcal{E} \to \mathcal{E}'$ of rank $r$ and $h': \mathcal{E}' \to \mathcal{E}$ of rank $n-r$ satisfying the conditions in (ii). Let
$\mathcal{E}'_1:=\mathrm{Im}(h') \subseteq \mathcal{E}$ (resp., $\mathcal{E}_1:=\mathrm{Im}(h) \subseteq \mathcal{E}'$) and $\mathcal{E}_2 \subseteq \mathcal{E}$ (resp., $\mathcal{E}_2' \subseteq \mathcal{E}'$) be the subbundle generated by $\mathcal{E}'_1$ (resp., $\mathcal{E}_1$). The conditions $h\circ h'=0= h'\circ h$ yield the following commutative diagram with exact rows and columns
\begin{equation}\label{sad}
\begin{tikzcd}
& 0 \\
& \mathcal{G}_1 \ar[u] & & 0 \ar[d] \\
0 \ar[r] & \mathcal{E}_2 \ar[r] \ar[u] & \mathcal{E} \ar[r] \ar[d,"h",shift left] & \mathcal{E}_1 \ar[r] \ar[d,"i"] & 0  \\
0 & \mathcal{E}'_1 \ar[l] \ar[u,"i'"] & \mathcal{E}' \ar[l] \ar[u,shift left,"h'"] & \mathcal{E}'_2 \ar[l] \ar[d] & 0 \ar[l] \\
& 0 \ar[u] & & \mathcal{G}_2 \ar[d] \\
& & & 0,
\end{tikzcd}
\end{equation}
where $\mathcal{G}_1:=\mathcal{E}_2/\mathcal{E}'_1$ and $\mathcal{G}_2:=\mathcal{E}'_2/\mathcal{E}_1$ are torsion sheaves of rank $0$. By (i) and commutativity of \eqref{comm-up} it follows that $\mathcal{G}_1 \cong \mathcal{G}_2$ and therefore
\[
0 \subseteq \mathcal{E}'_1 \subseteq \mathcal{E}_2 \subseteq \mathcal{E} \text{ and } \mathcal{E}' \supseteq \mathcal{E}'_2 \supseteq \mathcal{E}_1 \supseteq 0
\]
are opposite filtrations of $\mathcal{E}$ and $\mathcal{E}'$. Indeed, an easy computation shows
\begin{gather*}
\mathcal{G}_2 \cong \mathrm{coker}\left(\det(\mathcal{E}_1) \xrightarrow{\det(i)} \det(\mathcal{E}'_2)\right), \text{ and} \\
\mathcal{G}_1 \cong \mathrm{coker}\left(\det(\mathcal{E}_2)^{-1} \otimes \det(\mathcal{E}) \xrightarrow{\det(i')^{-1} \otimes \lambda} \det(\mathcal{E}'_1)^{-1} \otimes \det(\mathcal{E}')\right).
\end{gather*}
Furthermore, using \eqref{sad} we can decompose \eqref{comm-up} as follows
\[
\begin{tikzcd}
& \det(\mathcal{E}_1) \ar[rr,hook,"\det(i)"] \ar[dd,"\cong"',near start] & & \det(\mathcal{E}'_2) \ar[dd,"\cong"] \ar[dl] \\
\wedge^r(\mathcal{E}) \ar[rr] \ar[dd,"\cong"'] \ar[ur] & & \wedge^r(\mathcal{E}') \ar[dd,"\cong",near start] \\
& \det(\mathcal{E}_2)^{-1} \otimes \det(\mathcal{E}) \ar[rr,hook,"\det(i')^{-1} \otimes \lambda",near start] & & \det(\mathcal{E}'_1)^{-1} \otimes \det(\mathcal{E}') \ar[dl] \\
\wedge^{n-r}(\mathcal{E})^{-1} \otimes \det(\mathcal{E}) \ar[rr] \ar[ur] & & \wedge^{n-r}(\mathcal{E}')^{-1} \otimes \det(\mathcal{E}').
\end{tikzcd}
\]
This gives rise to the following commutative diagram with exact rows
\[
\begin{tikzcd}
0 \ar[r]& \det(\mathcal{E}_1) \ar[rr,"\det(i)"] \ar[d,"\cong"'] & & 
\det(\mathcal{E}'_2) \ar[d,"\cong"] \ar[r] & \mathcal{G}_2 \ar[r] \ar[d] & 0 \\
0 \ar[r]& \det(\mathcal{E}_2)^{-1} \otimes \det(\mathcal{E}) \ar[rr,"\det(i')^{-1} \otimes \lambda"] \ar[d,equal] & & \det(\mathcal{E}'_1)^{-1} \otimes \det(\mathcal{E}') \ar[r] & \mathcal{G} \ar[r] & 0\\
0\ar[r]& \det(\mathcal{E}_2)^{-1} \otimes \det(\mathcal{E}) \ar[rr,"\det(i')^{-1} \otimes \mathrm{id}"'] & & \det(\mathcal{E}'_1)^{-1} \otimes \det(\mathcal{E}) \ar[r] \ar[u,"\mathrm{id} \otimes \lambda","\cong"'] & \mathcal{G}_1 \ar[r] \ar[u] & 0.
\end{tikzcd}
\]
Hence, $\mathcal{G}_1 \cong \mathcal{G} \cong \mathcal{G}_2$. 
\end{proof}
\begin{cor}\label{1738e}
Let $\mathbf{K}/k$ be a field. Two points $\mathcal{E} \ncong \mathcal{E}' \in \mathscr{B}un_n^d(\mathbf{K})$ have opposite filtrations if there exist an isomorphism $\det(\mathcal{E}) \xrightarrow{\sim} \det(\mathcal{E}')$, and a morphism $\mathcal{E} \to \mathcal{E}'$ of rank $n-1$.
\end{cor}
\begin{proof}
Let $h: \mathcal{E} \to \mathcal{E}'$ be a morphism of rank $n-1$. Using the condition that (\ref{comm-up}) should commute we can define a rank $1$ morphism $h': \mathcal{E}' \to \mathcal{E}$, which satisfies $h \circ h'=h' \circ h=0$ and we are done by Proposition \ref{173e}.
\end{proof}
Another piece of information from the proof of Proposition \ref{prop:S-comple} is that coherent sheaves with opposite filtrations have a common generalization. This can be seen a non-separation phenomena in the topological sense and we formalize an algebraic analogue of it.
\begin{defn}
Two points $\mathcal{E} \ncong \mathcal{E}' \in \mathscr{C}oh_n^d(\mathbf{K})$ for some field $\mathbf{K}/k$ are said to be \emph{non-separated} if there exist
\begin{enumerate}
\item
a DVR $R$ over $k$ with fraction field $K$ and residue field $\mathbf{K}$, and
\item
two families $\mathcal{G}_R,\mathcal{G}'_R \in \mathscr{C}oh_n^d(R)$ with $\mathcal{G}_K \cong \mathcal{G}'_K$ such that 
\[
\mathcal{G}_\mathbf{K} \cong \mathcal{E} \text{ and } \mathcal{G}'_\mathbf{K} \cong \mathcal{E}'.
\]
\end{enumerate}
In this case we say the non-separation of $\mathcal{E}$ and $\mathcal{E}'$ is realized by $\mathcal{G}_R,\mathcal{G}'_R$.
\end{defn}
It turns out that for coherent sheaves, having opposite filtration and being non-separated are equivalent. This gives a geometric description of coherent sheaves with opposite filtrations.
\begin{cor}\label{1813}
Let $\mathbf{K}/k$ be a field. For any two points $\mathcal{E} \ncong \mathcal{E}' \in \mathscr{C}oh_n^d(\mathbf{K})$, the following are equivalent:
\begin{enumerate}
\item
$\mathcal{E}$ and $\mathcal{E}'$ have opposite filtrations.
\item
$\mathcal{E}$ and $\mathcal{E}'$ are non-separated.
\end{enumerate}
\end{cor}
\begin{proof}
(i) $\Rightarrow$ (ii): Suppose $\mathcal{E} \ncong \mathcal{E}' \in \mathscr{C}oh_n^d(\mathbf{K})$ have opposite filtrations $\mathcal{E}^\bullet,\mathcal{E}'_\bullet$. Consider the morphism $f: \overline{\mathrm{ST}}_R \to \mathscr{C}oh_n^d$ constructed in the proof of the \textsc{Only If Part} in Proposition \ref{prop:S-comple} using these two opposite filtrations. The non-separatedness of $\mathcal{E}$ and $\mathcal{E}'$ is realized by the families corresponding to $f|_{x \neq 0}$ and $f|_{y \neq 0}$.

(ii) $\Rightarrow$ (i): Suppose the non-separatedness of $\mathcal{E}$ and $\mathcal{E}'$ is realized by $\mathcal{G}_R,\mathcal{G}'_R$. This defines a morphism $\overline{\mathrm{ST}}_R - \{0\} \to \mathscr{C}oh_n^d$ extending to $f: \overline{\mathrm{ST}}_R \to \mathscr{C}oh_n^d$. By the proof of
the \textsc{If Part} in Proposition \ref{prop:S-comple}, the restrictions $f|_{x=0}$ and $f|_{y=0}$ define opposite filtrations of $\mathcal{E}$ and $\mathcal{E}'$.
\end{proof}
\begin{rmk}
The equivalence established in Corollary \ref{1813} relates us to several non-separation results in the literature. For example, Lemma \ref{2235} gives an algebraic proof of \cite[Proposition 2 and Corollary]{MR0529671} or \cite[Proposition 2.9]{MR0879544}, and Proposition \ref{173e} (resp., Corollary \ref{1738e}) gives an algebraic proof of \cite[Proposition 2]{MR0508172} (resp., \cite[Corollary 1]{MR0508172}).
\end{rmk}
\subsection{Local reductivity}
In positive characteristic the existence criterion of moduli spaces requires an additional condition called local reductivity, which we recall here.
\begin{defn}[(\cite{MR4665776}, Definition 2.1)]
If $\mathscr{X}$ is an algebraic stack and $x \in |\mathscr{X}|$ is a point, an \emph{\'{e}tale quotient presentation} around $x$ is a pointed \'{e}tale morphism $f: (\mathscr{W},w) \to (\mathscr{X},x)$ of algebraic stacks such that (i) $\mathscr{W} \cong [\mathrm{Spec}(A)/\mathrm{GL}_n]$ for some $n$; and (ii) $f$ induces an isomorphism of stabilizer groups at $w$, i.e., $\mathrm{Aut}_{\mathscr{W}}(p)=\mathrm{Aut}_{\mathscr{X}}(f(p))$ for any point $p \in \mathscr{W}(\mathbf{K})$ representing $w$, where $\mathbf{K}/k$ is a field.
\end{defn}
\begin{rmk}\label{rmk:G-GLn}
The notion of \'{e}tale quotient presentation remains unchanged if we require in (i) that $\mathscr{W} \cong [\mathrm{Spec}(A)/G]$ for some reductive group $G$, as for any closed embedding $G \to \mathrm{GL}_n$ we have $[\mathrm{Spec}(A)/G] \cong [((\mathrm{Spec}(A) \times \mathrm{GL}_n)/G)/\mathrm{GL}_n]$ and $(\mathrm{Spec}(A) \times \mathrm{GL}_n)/G$ is affine since $G$ is reductive.
\end{rmk}
\begin{defn}[(\cite{MR4665776}, Definition 2.5)]
A quasi-separated algebraic stack $\mathscr{X}$ with affine stabilizers is \emph{locally reductive} if every point of $\mathscr{X}$ specializes to a closed point and every closed point admits an \'{e}tale quotient presentation.
\end{defn}

Recall that in characteristic $0$ having a separated moduli space implies (actually is equivalent to) $\Theta$-reductivity and S-completeness (see \cite[Theorem A]{MR4665776}). For principal bundles this implication remains valid in positive characteristic.
\begin{prop}\label{lem:localreductive}
Let $C$ be a smooth projective connected curve over an algebraically closed field $k$ and let $G$ be a reductive group over $k$. Then a quasi-compact open substack $\mathscr{U} \subseteq \mathscr{B}un_G$ admits a separated adequate moduli space if and only if it is locally reductive, $\Theta$-reductive, and S-complete.
\end{prop}
\begin{proof}
Without loss of generality we may assume that $\mathscr{U} \subseteq \mathscr{B}un_G^d$ for some $d \in \pi_1(G)$. By Theorem \ref{thm0}, it remains to prove that if $\mathscr{U} \subseteq \mathscr{B}un_G^d$ admits an adequate moduli space (not necessarily separated) $\mathscr{U} \to U$, then $\mathscr{U}$ is locally reductive. This follows if $\mathscr{U}$ is a quotient stack for a reductive group action. Indeed, if $\mathscr{U}=[X/H]$ for a reductive group $H$ acting on an algebraic space $X$, then it suffices to show that \'{e}tale locally around any point $x \in X$ with closed $H$-orbit there is an $H$-invariant open affine neighbourhood. Let $u \in U$ be the image of $x$ under the composition $X \to \mathscr{U} \to U$ and let $V \to U$ be an \'{e}tale affine neighbourhood of $u$. Note that the morphism $X \to \mathscr{U}$ is affine.
Then the composition $X \to \mathscr{U} \to U$ is adequately affine and hence affine (see \cite[Theorem 4.3.1]{MR3272912}). Thus, $X \times_U V \to X$ is an $H$-invariant \'{e}tale affine neighbourhood of $x$ and we are done by Remark \ref{rmk:G-GLn}.

Therefore it remains to show $\mathscr{U}$ is a quotient stack for a reductive group action. Choose a closed embedding $\phi: G \to \mathrm{GL}_n$. This induces a representable morphism $\phi_*: \mathscr{B}un_{G}^d \to \mathscr{B}un_n^{e}$ (see \cite[Fact 2.3]{MR3013030}), where $e=\phi_*(d) \in \mathbf{Z}$. Since $\mathscr{U} \subseteq \mathscr{B}un_G^d$ is quasi-compact and $\mathscr{B}un_n^e$ has an exhaustive filtration given by quotient stacks for $\mathrm{GL}_N$-actions (see, e.g. \cite[Lemma 4.1.5 and 4.1.11]{Wang2011}), the restriction $\phi_*|_{\mathscr{U}}: \mathscr{U} \to \mathscr{B}un_n^e$ factors through such a quotient stack. As any algebraic stack admitting a representable morphism to a quotient stack is also a quotient stack for an action of the same group, we are done.
\end{proof}
\begin{rmk}\label{1505-5}
The $\Theta$-reductivity and S-completeness of $\mathscr{U}$ can be verified directly and uniformly. Let $R$ be a DVR over $k$ and let $\mathscr{X}$ be $\Theta_R$ or $\overline{\mathrm{ST}}_R$. For any morphism $\mathscr{X}-\{0\} \to \mathscr{U}$ by the universal property of the adequate moduli space (see \cite[Theorem 3.12]{Alper2019-1}) the composition $\mathscr{X}-\{0\} \to \mathscr{U} \to U$ factors through $\mathscr{X}-\{0\} \to \mathrm{Spec}(R) \to U$ (the same would hold for the composition of any extension $\mathscr{X} \to \mathscr{U} \to U$) since $U$ is separated. Then we can base change to $\mathrm{Spec}(R)$ and instead prove $\Theta$-reductivity and S-completeness of the fiber product $\mathscr{U} \times_U \mathrm{Spec}(R) \to \mathrm{Spec}(R)$. In the proof of Proposition \ref{lem:localreductive} we saw that $\mathscr{U} \cong [X/\mathrm{GL}_{N}]$ is a quotient stack and hence $\mathscr{U} \times_U \mathrm{Spec}(R) \cong [X_R/\mathrm{GL}_{N,R}]$. Since $X_R$ is an affine scheme and $\mathrm{GL}_{N,R}$ is a reductive group scheme over $R$, we conclude by \cite[Proposition 3.21 (2) and Proposition 3.44 (2)]{MR4665776}. 
\end{rmk}
\section{Rank 2}\label{trans-vb2}
Both $\Theta$-reductivity and S-completeness have simple consequences in rank 2. This allows us to apply Proposition \ref{lem:localreductive} and obtain the following.
\begin{thm}\label{1724}
Let $C$ be a smooth projective connected curve of genus $g_C>1$ over an algebraically closed field $k$. Then the open substack $\mathscr{B}un_2^{d,ss} \subseteq \mathscr{B}un_2^d$ of semistable vector bundles is the unique maximal open substack that admits a separated adequate moduli space.
\end{thm}
The road map is: If $\mathscr{U} \subseteq \mathscr{B}un_2^d$ is an open substack, then
\begin{itemize}
\item
If $\mathscr{U}$ is $\Theta$-reductive, then it cannot contain any direct sum of line bundles of different degrees (Proposition \ref{24}).
\item
If $\mathscr{U}$ is S-complete, then it contains some direct sum of line bundles of different degrees, provided it contains an unstable vector bundle (Proposition \ref{1704-beg} and Lemma \ref{23}).
\end{itemize}
Thus, to obtain a separated adequate moduli space we must have $\mathscr{U} \subseteq \mathscr{B}un_2^{d,ss}$.
\subsection{More on $\Theta$-reductivity}
Recall that any open substack of $\mathscr{B}un_n^d$ supporting a 
$\Theta$-testing family cannot be $\Theta$-reductive (see Corollary \ref{0305-1}). This fact has a simple implication in rank 2.
\begin{prop}\label{24}
Let $\mathscr{U} \subseteq \mathscr{B}un_2^d$ be a $\Theta$-reductive open substack. Then $\mathscr{U}$ cannot contain any direct sum of line bundles of different degrees.
\end{prop}
To prove this we need the following technical lemma.
\begin{lem}\label{1801}
Let $\mathscr{U} \subseteq \mathscr{B}un_2^d$ be an open substack. If $\mathcal{V}_1 \oplus \mathcal{V}_2 \in \mathscr{U}(k)$ for some line bundles $\mathcal{V}_1,\mathcal{V}_2$ of degrees $d_1 > d_2$ respectively, then for any effective divisor $D$ on $C$ of degree $d_1-d_2$, there exist line bundles $\mathcal{L}_1,\mathcal{L}_2$ of degrees $d_1,d_2$ respectively such that
\[
\mathcal{L}_1 \oplus \mathcal{L}_2 \in \mathscr{U}(k) \text{ and } \mathcal{L}_2(D) \oplus \mathcal{L}_1(-D) \in \mathscr{U}(k).
\]
\end{lem}
\begin{proof}
To start we note that any effective divisor $D$ on $C$ of degree $d_1-d_2$ induces an automorphism of $\mathrm{Pic}^{d_1} \times \mathrm{Pic}^{d_2}$:
\[
\lambda_D: \mathrm{Pic}^{d_1} \times \mathrm{Pic}^{d_2} \xrightarrow{\sim} \mathrm{Pic}^{d_1} \times \mathrm{Pic}^{d_2} \text{ mapping } (\mathcal{L}_1,\mathcal{L}_2) \mapsto (\mathcal{L}_2(D),\mathcal{L}_1(-D)).
\]
Let $\oplus: \mathrm{Pic}^{d_1} \times \mathrm{Pic}^{d_2} \to \mathscr{B}un_2^d$ be the direct sum morphism. The open subset $\Omega:=\oplus^{-1}(\mathscr{U}) \subseteq \mathrm{Pic}^{d_1} \times \mathrm{Pic}^{d_2}$ is non-empty as it contains $(\mathcal{V}_1,\mathcal{V}_2)$, which implies that $\lambda_D^{-1}(\Omega) \cap \Omega \neq \emptyset$. Then any point $(\mathcal{L}_1,\mathcal{L}_2)$ in this intersection satisfies the required conditions.
\end{proof}
\begin{proof}[Proof of Proposition $\ref{24}$]
Suppose  $\mathcal{V}_1 \oplus \mathcal{V}_2 \in \mathscr{U}(k)$ for some line bundles $\mathcal{V}_1,\mathcal{V}_2$ of degrees $d_1>d_2$ respectively. Fix an effective divisor $D$ on $C$ of degree $d_1-d_2$. By Lemma \ref{1801} there exist line bundles $\mathcal{L}_1,\mathcal{L}_2$ of degree $d_1,d_2$ respectively such that 
\[
\mathcal{L}_1 \oplus \mathcal{L}_2 \in \mathscr{U}(k) \text{ and } \mathcal{L}_2(D) \oplus \mathcal{L}_1(-D) \in \mathscr{U}(k).
\]
Let $[\varrho] \in \mathrm{Ext}^1(\mathcal{O}_D \oplus \mathcal{L}_2,\mathcal{L}_1(-D))$ be the extension class represented by
\[
0 \to \mathcal{L}_1(-D) \xrightarrow{(\text{incl},0)} \mathcal{L}_1 \oplus \mathcal{L}_2 \to \mathcal{O}_D \oplus \mathcal{L}_2 \to 0.
\]
By definition the quadruple
\[
(\mathcal{L}_2(D),\mathcal{L}_1(-D),D,[\varrho]) \in \mathrm{Pic}^{d_1}(k) \times \mathrm{Pic}^{d_2}(k) \times \mathrm{Div}^{\mathrm{eff}}(C) \times \mathrm{Ext}^1(\mathcal{O}_D \oplus \mathcal{L}_2,\mathcal{L}_1(-D))
\]
satisfies $\mathcal{L}_2(D) \oplus \mathcal{L}_1(-D) \in \mathscr{U}(k)$ and $\mathcal{E}([\varrho])=\mathcal{L}_1 \oplus \mathcal{L}_2 \in \mathscr{U}(k)$. Then $\mathscr{U}$ is not $\Theta$-reductive by Corollary \ref{0305-1}, a contradiction.
\end{proof}
\subsection{More on S-completeness}
Recall that S-completeness can be described in terms of vector bundles with opposite filtrations (see Proposition \ref{prop:S-comple}), which in rank 2 are simply the ones with the same determinant and non-zero morphisms between them.
\begin{lem}\label{172e}
Let $\mathbf{K}/k$ be a field.
Two points $\mathcal{E} \ncong \mathcal{E}' \in \mathscr{B}un_2^d(\mathbf{K})$ have opposite filtrations if and only if there exist an isomorphism $\det(\mathcal{E}) \xrightarrow{\sim} \det(\mathcal{E}')$, and a non-zero morphism $\mathcal{E} \to \mathcal{E}'$.
\end{lem}
\begin{proof}
\textsc{Only If Part}: Lemma \ref{2235}. \textsc{If Part}: Corollary \ref{1738e}.
\end{proof}
Similarly, S-completeness also has a simple implication in rank $2$.
\begin{lem}\label{23}
Let $\mathscr{U} \subseteq \mathscr{B}un_2^d$ be an S-complete open substack. If $\mathbf{K}/k$ is a field and $\mathcal{E} \ncong \mathcal{E}' \in \mathscr{U}(\mathbf{K})$ have opposite filtrations with $\mathcal{E}$ unstable, then $\mathscr{U}$ contains a direct sum of line bundles of different degrees.
\end{lem}
\begin{proof}
If $\mathcal{E}^{\bullet},\mathcal{E}'_{\bullet}$ are opposite filtrations of $\mathcal{E}$ and $\mathcal{E}'$, then by Proposition \ref{prop:S-comple} we have $\mathrm{gr}(\mathcal{E}^{\bullet}) \in \mathscr{U}(\mathbf{K})$, which is unstable since $\mathcal{E}$ is.
\end{proof}

\subsection{Proof of Theorem $\ref{1724}$}
This can be deduced from the following result.
\begin{prop}\label{1704-beg}
Let $\mathscr{U} \subseteq \mathscr{B}un_2^d$ be an open substack containing an unstable vector bundle. Then there exist two points in $\mathscr{U}$ with opposite filtrations such that one of them is unstable.
\end{prop}
The argument is essentially from the proof of \cite[Proposition 4]{MR0508172}, i.e., we can produce such two points via flipping the Harder-Narasimhan filtration of the unstable vector bundle. To achieve this we need the following technical lemma, motivated by \cite[\S 2]{MR250496}.
\begin{lem}\label{1943-2}
Let $\mathscr{U} \subseteq \mathscr{B}un_2^d$ be an open substack containing an unstable vector bundle $\mathcal{E}$ with Harder-Narasimhan filtration $0 \to \mathcal{L}_1 \to \mathcal{E} \to \mathcal{L}_2 \to 0$. Define
\[
\Delta:=\{\mathcal{L} \in \mathrm{Pic}^0: \exists [e] \in \mathrm{Ext}^1(\mathcal{L}^{-1} \otimes \mathcal{L}_2,\mathcal{L} \otimes \mathcal{L}_1) \text{ such that } \mathcal{E}([e]) \in \mathscr{U}\},
\]
where $\mathcal{E}([e])$ is the extension vector bundle associated to $[e]$. Then $\Delta \subseteq \mathrm{Pic}^0$ is open and dense.
\end{lem}
\begin{proof}
By assumption $\Delta$ is non-empty as it contains $\mathcal{O}_C$. It remains to show $\Delta \subseteq \mathrm{Pic}^0$ is open. 

Let $d_i:=\deg(\mathcal{L}_i)$ for $i=1,2$ and let $\vartheta_1: \mathrm{Pic}^0 \to \mathrm{Pic}^{d_1}$ (resp., $\vartheta_2: \mathrm{Pic}^0 \to \mathrm{Pic}^{d_2}$) be the morphism given by $\mathcal{L} \mapsto \mathcal{L} \otimes \mathcal{L}_1$ (resp., $\mathcal{L} \mapsto \mathcal{L}^{-1} \otimes \mathcal{L}_2$). Denote by $\mathcal{P}_i \to C \times \mathrm{Pic}^{d_i}$ the universal family and $\mathcal{Q}_i:=(\mathrm{id}_C \times \vartheta_i)^*\mathcal{P}_i$ its pull-back to $C \times \mathrm{Pic}^0$. The two families define a Cartesian diagram (see Lemma \ref{2321})
\[
\begin{tikzcd}
\mathscr{E}xt(\mathcal{Q}_2,\mathcal{Q}_1) \ar[d,"{\mathrm{pr}_{13}}"'] \ar[r,"\psi"]& \mathscr{E}xt(d_2,d_1) \ar[d,"{\mathrm{pr}_{13}}"] \ar[r,"{\mathrm{pr}_2}"] & \mathscr{C}oh_2^d \\
\mathrm{Pic}^0 \ar[r,"{(\mathcal{Q}_1,\mathcal{Q}_2)}"'] & \mathscr{C}oh_1^{d_1} \times \mathscr{C}oh_1^{d_2}. \ar[ul,phantom,"\lrcorner"]
\end{tikzcd}
\]
Then $\mathrm{pr}_{13}^{-1}(\mathcal{L})=\mathscr{E}xt(\mathcal{Q}_2,\mathcal{Q}_1)|_{\mathcal{L}}=[\mathrm{Ext}^1(\mathcal{L}^{-1} \otimes \mathcal{L}_2,\mathcal{L} \otimes \mathcal{L}_1)/\mathrm{Hom}(\mathcal{L}^{-1} \otimes \mathcal{L}_2,\mathcal{L} \otimes \mathcal{L}_1)]$ for any $\mathcal{L} \in \mathrm{Pic}^0$. By definition $\Delta=\mathrm{pr}_{13} \circ (\mathrm{pr}_2\circ \psi)^{-1}(\mathscr{U})$ which is open in $\mathrm{Pic}^0$ since $\mathrm{pr}_{13}$ is smooth.
\end{proof}
\begin{proof}[Proof of Proposition $\ref{1704-beg}$]
Let $\mathcal{E} \in \mathscr{U}(k)$ be an unstable vector bundle with Harder-Narasimhan filtration $0 \to \mathcal{L}_1 \to \mathcal{E} \to \mathcal{L}_2 \to 0$. Fix a line bundle $\beta \in \mathrm{Pic}^{d_1-d_2}$ admitting a non-zero global section where $d_i:=\deg(\mathcal{L}_i)$. Consider the open dense subset $\Delta \subseteq \mathrm{Pic}^0$ defined in Lemma \ref{1943-2}. Then the intersection $\Delta \cap (\mathcal{L}_1 \otimes \mathcal{L}_2^{-1} \otimes \beta^{-1} \otimes \Delta)^{-1}$ is non-empty. 
For any point $\mathcal{L}$ in this intersection, i.e., $\mathcal{L}\in \Delta$ and $\mathcal{L}^{-1} \otimes \mathcal{L}_1^{-1} \otimes \mathcal{L}_2 \otimes \beta \in \Delta$, by definition there exist $[x] \in \mathrm{Ext}^1(\mathcal{L}^{-1} \otimes \mathcal{L}_2,\mathcal{L} \otimes \mathcal{L}_1)$ and $[y] \in \mathrm{Ext}^1(\mathcal{L} \otimes \mathcal{L}_1 \otimes \beta^{-1},\mathcal{L}^{-1} \otimes \mathcal{L}_2 \otimes \beta)$ such that both $\mathcal{E}([x])$ and $\mathcal{E}([y])$ lie in $\mathscr{U}$, i.e.,
\begin{gather*}
0 \to \mathcal{L} \otimes \mathcal{L}_1 \to \mathcal{E}([x]) \to \mathcal{L}^{-1} \otimes \mathcal{L}_2 \to 0 \\
0 \to \mathcal{L}^{-1} \otimes \mathcal{L}_2 \otimes \beta \to \mathcal{E}([y]) \to \mathcal{L} \otimes \mathcal{L}_1 \otimes \beta^{-1} \to 0.
\end{gather*}
Then $\mathcal{E}([x])$ and $\mathcal{E}([y])$ are both unstable with the same determinant, and there exist non-zero morphisms between them. By Lemma \ref{172e} they have opposite filtrations and we are done.
\end{proof}
\subsection{Generalizing to arbitrary base field}\label{sub:generalizing}
The maximality type results Theorem \ref{1535-2} and \ref{thmA} hold over an arbitrary field $k$ since we can reduce to the case that the base field is algebraically closed. To indicate the base curve we temporarily use $\mathscr{B}un_G(C)$ instead of $\mathscr{B}un_G$. Fix an algebraic closure $\bar{k}$ of $k$.
\begin{lem}
\label{lem-arbitrary}
Let $C$ be a smooth projective geometrically connected curve over a field $k$ and let $G$ be a geometrically connected reductive group over $k$. If the open substack $\mathscr{B}un_{G_{\bar{k}}}^{ss}(C_{\bar{k}}) \subseteq \mathscr{B}un_{G_{\bar{k}}}(C_{\bar{k}})$ of semistable principal $G_{\bar{k}}$-bundles over $C_{\bar{k}}$ is the unique maximal open substack that admits a (schematic or separated) adequate moduli space, then so is the open substack $\mathscr{B}un_{G}^{ss}(C) \subseteq \mathscr{B}un_{G}(C)$ of semistable principal $G$-bundles over $C$.
\end{lem}
\begin{proof}
If $\mathscr{U} \subseteq \mathscr{B}un_G(C)$ is an open substack that admits a (schematic or separated) adequate moduli space, then so is the base change $\mathscr{U}_{\bar{k}} \subseteq \mathscr{B}un_{G}(C)_{\bar{k}}=\mathscr{B}un_{G_{\bar{k}}}(C_{\bar{k}})$ by \cite[Proposition 5.2.9 (1)]{MR3272912}. In a diagram:
\[
    \begin{tikzcd}
        \mathscr{U}_{\bar{k}} \ar[r] \ar[d,"\text{ams}"'] & \mathscr{U} \ar[d,"\text{ams}"] \\
        U_{\bar{k}} \ar[r] \ar[d] & U \ar[ul,phantom,"\lrcorner"] \ar[d] \\
        \mathrm{Spec}(\bar{k}) \ar[r] & \mathrm{Spec}(k) \ar[ul,phantom,"\lrcorner"].
    \end{tikzcd}
\]
By assumption, we obtain $\mathscr{U}_{\bar{k}} \subseteq \mathscr{B}un^{ss}_{G_{\bar{k}}}(C_{\bar{k}})$. As a $G$-bundle $\mathcal{P}$ over $C$ is semistable if the base change $\mathcal{P}_{\bar{k}}$ over $C_{\bar{k}}$ is semistable, this shows that $\mathscr{U} \subseteq \mathscr{B}un_G^{ss}(C)$.  
\end{proof}
\section{Higher rank}\label{eg-vb3}
In this section, we construct an open substack $\mathscr{U} \subseteq \mathscr{B}un_3^2$ that admits a separated non-proper adequate moduli space and is not contained in $\mathscr{B}un_3^{2,ss}$. 

Throughout this section we assume that $g_C>2$. This guarantees the existence of $(1,0)$-stable vector bundles over $C$ of arbitrary rank and degree (see \cite[Proposition 5.4 (i)]{MR541029}). Recall \cite[Definition 5.1]{MR541029}: for any pair of non-negative integers $(i,j)$, a vector bundle $\mathcal{E}$ over $C$ is said to be $(i,j)$\emph{-stable} if for any non-zero proper subbundle $0 \neq \mathcal{F} \subsetneq \mathcal{E}$ we have
\[
\dfrac{\deg(\mathcal{F})+i}{\mathrm{rk}(\mathcal{F})} < \dfrac{\deg(\mathcal{E})+i-j}{\mathrm{rk}(\mathcal{E})}.
\]
\subsection{Construction of $\mathscr{U} \subseteq \mathscr{B}un_3^{2}$}
Let $\lambda$ be the polygon in the rank-degree plane consisting of vertices $\{(0,0),(1,1),(3,2)\}$ and denote by $\mathscr{B}un_3^{2,\leq \lambda} \subseteq \mathscr{B}un_3^2$ the open substack consisting of vector bundles whose Harder-Narasimhan polygons lie below $\lambda$.
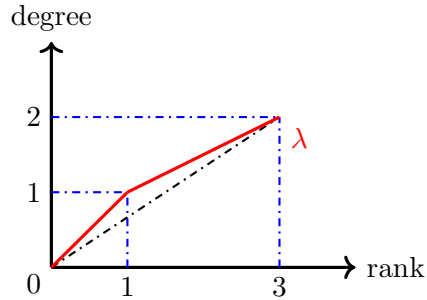
\begin{figure}[!ht]
\centering
\begin{tikzpicture}
\draw[->,very thick] (0,0) -- (4,0) node[right] {\text{rank}};
\draw[->,very thick] (0,0) -- (0,3) node[above] {\text{degree}};

\draw[thick,dash dot,color=blue] (0,1) -- (1,1);
\draw[thick,dash dot,color=blue] (1,0) -- (1,1);
\draw[thick,dash dot,color=blue] (0,2) -- (3,2);
\draw[thick,dash dot,color=blue] (3,0) -- (3,2);
\draw[thick,dash dot] (0,0) -- (3,2);

\draw[color=red,very thick] (0,0) -- (1,1);
\draw[color=red,very thick] (1,1) -- (3,2) node[below right] {$\lambda$};

\node[below=6pt,left] at (0,0) {$0$};
\node[left] at (0,1) {$1$};
\node[left] at (0,2) {$2$};
\node[below] at (1,0) {$1$};
\node[below] at (3,0) {$3$};
\end{tikzpicture}
\caption{The polygon $\lambda$ in the rank-degree plane}
\end{figure}
\begin{lem}\label{no}
Suppose $\mathcal{E} \in \mathscr{B}un_3^{2,\leq \lambda}(\mathbf{K})$ for some field $\mathbf{K}/k$ and $\mathcal{F} \subseteq \mathcal{E}$ is a subsheaf of rank at most $2$. Then $\deg(\mathcal{F}) \leq 1$. In particular, if $\deg(\mathcal{F})=1$, then $\mathcal{F} \subseteq \mathcal{E}$ is a subbundle.
\end{lem}
\begin{proof}
If $\deg(\mathcal{F})>1$, then the Harder-Narasimhan polygon of $\mathcal{E}$ lies strictly above $\lambda$.
\end{proof}
The desired open substack $\mathscr{U} \subseteq \mathscr{B}un_3^{2}$ will be an open substack of $\mathscr{B}un_3^{2,\leq \lambda}$. Since we want it to be essentially different from $\mathscr{B}un_3^{2,ss}$, we need to include some unstable vector bundles. However, not every unstable vector bundle in $\mathscr{B}un_3^{2,\leq \lambda}$ is allowed. The S-completeness gives a first restriction: those points $\mathcal{E} \in \mathscr{B}un_3^{2,\leq \lambda}$ with Harder-Narasimhan filtration $0 \to \mathcal{L} \to \mathcal{E} \to \mathcal{F} \to 0$ such that $\mathrm{Hom}(\mathcal{F},\mathcal{L}) \neq 0$ should be removed. Then $\Theta$-reductivity also removes those points $\mathcal{E} \in \mathscr{B}un_3^{2,\leq \lambda}$ fitting into a short exact sequence $0 \to \mathcal{F} \to \mathcal{E} \to \mathcal{L} \to 0$ such that $\mathrm{Hom}(\mathcal{F},\mathcal{L}) \neq 0$. To retain $\Theta$-reductivity in this new substack, it suffices to remove, in both situations, those points $\mathcal{E} \in \mathscr{B}un_3^{2,\leq \lambda}$ such that $\mathcal{F}$ is not $(1,0)$-stable. Note that $\mathrm{Hom}(\mathcal{F},\mathcal{L}) \neq 0$ implies that $\mathcal{F}$ is not $(1,0)$-stable as any non-zero morphism from $\mathcal{F}$ to $\mathcal{L}$ would yield a sub-line bundle of degree $\geq 0$. In summary we should cut out the following subset
\[
\mathscr{A}:=\left\langle \mathcal{E} \in |\mathscr{B}un_3^{2,\leq \lambda}|:
\begin{matrix}
\mathcal{E} \text{ fits into } 0 \to \mathcal{L} \to \mathcal{E} \to \mathcal{F} \to 0 \text{ or }\\
0 \to \mathcal{F} \to \mathcal{E} \to \mathcal{L} \to 0 \\
\text{for some } \mathcal{L} \in \mathscr{B}un_1^1 \text{ and } \mathcal{F} \in \mathscr{B}un_2^1-\mathscr{B}un_2^{1,s}(1,0) 
\end{matrix}
\right\rangle,
\]
where $\mathscr{B}un_2^{1,s}(1,0) \subseteq \mathscr{B}un_2^1$ denotes the open substack of $(1,0)$-stable vector bundles. Note that a rank $2$, degree $1$ vector bundle is $(1,0)$-stable if and only if all of its sub-line bundles have degrees $<0$. To start we show $\mathscr{A}$ is a closed subset.
\begin{lem}
The subset $\mathscr{A} \subseteq |\mathscr{B}un_3^{2,\leq \lambda}|$ is closed.
\end{lem}
\begin{proof}
The subset $\mathscr{A} \subseteq |\mathscr{B}un_3^{2,\leq \lambda}|$ is constructible, to conclude we claim it is stable under specialization. Let $R$ be a DVR over $k$ with fraction field $K$ and residue field $\kappa$. For any family $\mathcal{E}_R \in \mathscr{B}un_3^{2,\leq \lambda}(R)$ with $\mathcal{E}_K \in \mathscr{A}(K)$, we show $\mathcal{E}_\kappa \in \mathscr{A}(\kappa)$. If $0 \to \mathcal{F}_K \to \mathcal{E}_K \to \mathcal{L}_K \to 0$ is a defining sequence of $\mathcal{E}_K \in \mathscr{A}(K)$, we consider its degeneration $0 \to \mathcal{F}_\kappa \to \mathcal{E}_\kappa \to \mathcal{L}_\kappa \to 0$ and then $\mathcal{F}_\kappa \subseteq \mathcal{E}_\kappa$ is a subbundle by Lemma \ref{no}. Since $\mathcal{F}_K$ is not (1,0)-stable, neither is $\mathcal{F}_\kappa$. This shows that $\mathcal{E}_\kappa \in \mathscr{A}(\kappa)$. The same argument applies, if $\mathcal{E}_K \in \mathscr{A}(K)$ is realized by a sequence of the form $0 \to \mathcal{L}_K \to \mathcal{E}_K \to \mathcal{F}_K \to 0$. 
\end{proof}
Then we equip the closed subset $\mathscr{A} \subseteq |\mathscr{B}un_3^{2,\leq \lambda}|$ with the reduced induced algebraic stack structure so it becomes a closed substack of $\mathscr{B}un_3^{2,\leq \lambda}$. Let 
\[
\mathscr{U}:=\mathscr{B}un_3^{2,\leq \lambda}-\mathscr{A} \subseteq \mathscr{B}un_3^{2,\leq \lambda} \subseteq \mathscr{B}un_3^{2},
\]
i.e.,
\[
\mathscr{U}=\left\langle \mathcal{E} \in \mathscr{B}un_3^{2,\leq \lambda}:
\begin{matrix}
\text{If } \mathcal{E} \text{ fits into } 0 \to \mathcal{L} \to \mathcal{E} \to \mathcal{F} \to 0 \text{ or } 0 \to \mathcal{F} \to \mathcal{E} \to \mathcal{L} \to 0 \\
\text{for some } \mathcal{L} \in \mathscr{B}un_1^1 \text{ and } \mathcal{F} \in \mathscr{B}un_2^1, \text{ then } \mathcal{F} \text{ is } (1,0)\text{-stable.}
\end{matrix}
\right\rangle,
\]
and it is an open substack. To close we show that $\mathscr{U}$ contains unstable vector bundles.
\begin{lem}\label{dec}
Suppose $\mathcal{L} \in \mathscr{B}un_1^1(\mathbf{K})$ and $\mathcal{F} \in \mathscr{B}un_2^{1}(\mathbf{K})$ for some field $\mathbf{K}/k$. Then $\mathcal{L} \oplus \mathcal{F} \in \mathscr{U}(\mathbf{K})$ if and only if $\mathcal{F}$ is $(1,0)$-stable. This exhausts all decomposable vector bundles in $\mathscr{U}$. In particular, the open substack $\mathscr{U}$ is non-empty and $\mathscr{U} \nsubseteq \mathscr{B}un_3^{2,ss}$.
\end{lem}
\begin{proof}
As $\mathscr{B}un_3^{2,\leq \lambda}$ does not contain direct sums of line bundles, direct sums of the form $\mathcal{L} \oplus \mathcal{F}$ exhaust all decomposable vector bundles in $\mathscr{U}$ by Lemma \ref{no}.

\textsc{Only If Part}: This is clear from the definition of $\mathscr{U}$. \textsc{If Part}: By the uniqueness of the Harder-Narasimhan filtration $\mathcal{L}\oplus\mathcal{F}$ cannot be realized as a point in $\mathscr{A}$ via a sequence of the form $0 \to \mathcal{L}' \to \mathcal{L} \oplus \mathcal{F} \to \mathcal{F}' \to 0$ for some $\mathcal{L}' \in \mathscr{B}un_1^1(\mathbf{K})$ and $\mathcal{F}' \in \mathscr{B}un_2^1(\mathbf{K})$. If $\mathcal{L} \oplus \mathcal{F}$ fits into a sequence $0 \to \mathcal{F}' \to \mathcal{L} \oplus \mathcal{F} \to \mathcal{L}' \to 0$, then we consider the composition $\mathcal{L} \hookrightarrow \mathcal{L} \oplus \mathcal{F} \twoheadrightarrow \mathcal{L}'$.

If it is non-zero, then $\mathcal{L} \cong \mathcal{L}'$ and hence $\mathcal{F}' \cong \mathcal{F}$. This sequence does not realize $\mathcal{L} \oplus \mathcal{F}$ as a point in $\mathscr{A}$. If it is zero, then it factors through $\mathcal{L} \hookrightarrow \mathcal{F}'$ and hence $\mathcal{F}' \cong \mathcal{L} \oplus \mathcal{L}_0$ for some degree $0$ sub-line bundle $\mathcal{L}_0 \subseteq\mathcal{F}$, a contradiction. This shows that $\mathcal{L} \oplus \mathcal{F} \notin \mathscr{A}(\mathbf{K})$.
\end{proof}
\begin{cor}\label{d}
Suppose $\mathcal{E} \in \mathscr{B}un_3^{2,\leq \lambda}(\mathbf{K})$ for some field $\mathbf{K}/k$ fits into a short exact sequence 
$0 \to \mathcal{L} \to \mathcal{E} \to \mathcal{F} \to 0$ or $0 \to \mathcal{F} \to \mathcal{E} \to \mathcal{L} \to 0$
for some $\mathcal{L} \in \mathscr{B}un_1^1(\mathbf{K})$ and $\mathcal{F} \in \mathscr{B}un_2^1(\mathbf{K})$.
Then $\mathcal{E} \in \mathscr{U}(\mathbf{K})$ if and only if $\mathcal{L} \oplus \mathcal{F} \in \mathscr{U}(\mathbf{K})$, if and only if $\mathcal{F}$ is $(1,0)$-stable.
\end{cor}
\begin{proof}
\textsc{If Part}: This is clear, since $\mathscr{U}$ is open. \textsc{Only If Part}: If $\mathcal{L} \oplus \mathcal{F} \notin \mathscr{U}(\mathbf{K})$, then $\mathcal{F}$ is not $(1,0)$-stable by Lemma \ref{dec}. This shows that $\mathcal{E} \notin \mathscr{U}(\mathbf{K})$.
\end{proof}
The main result of this section is the following.
\begin{thm}\label{rank-3-main}
The open dense substack $\mathscr{U} \subseteq \mathscr{B}un_3^{2}$ admits a separated non-proper good moduli space. Furthermore, it is not contained in $\mathscr{B}un_3^{2,ss}$.
\end{thm}
\subsection{Proof of Theorem $\ref{rank-3-main}$}
To show that $\mathscr{U}$ admits a separated non-proper adequate moduli space, we check the conditions of Theorem \ref{thm0}.
\begin{lem}\label{local-theta}
The open substack $\mathscr{U}$ is locally reductive.
\end{lem}
\begin{proof}
By \cite[Remark 2.6]{MR4665776}, it suffices to show closed points of $\mathscr{U}$ have linearly reductive stabilizers. If $\mathcal{E} \in \mathscr{U}(\mathbf{K})$ is a closed point, then 
\begin{itemize}
\item 
If $\mathcal{E}$ is stable, then its stabilizer is $\mathbf{G}_m$.
\item 
If $\mathcal{E}$ is unstable, then it must be decomposable by Corollary \ref{d} (as it is closed), i.e., of the form $\mathcal{L} \oplus \mathcal{F}$ such that $\mathcal{L} \in \mathscr{B}un_1^1(\mathbf{K})$ and $\mathcal{F} \in \mathscr{B}un_2^1(\mathbf{K})$ is $(1,0)$-stable by Lemma \ref{dec}. In this case the stabilizer is $\mathbf{G}_m^2$ since there is no morphism between $\mathcal{L}$ and $\mathcal{F}$.
\end{itemize}
In either case, the stabilizer of $\mathcal{E}$ is linearly reductive.
\end{proof}
To verify $\Theta$-reductivity we need to understand the decomposable vector bundles in $\mathscr{U}$. They have been determined in Lemma \ref{dec}.
\begin{lem}\label{theta}
The open substack $\mathscr{U}$ is $\Theta$-reductive.
\end{lem}
\begin{proof}
Using Proposition \ref{thetaforopen1}. Let $R$ be a DVR over $k$ with fraction field $K$ and residue field $\kappa$. For any family $\mathcal{E}_R \in \mathscr{U}(R)$ and any (non-trivial) filtration $\mathcal{E}^\bullet_K$ of $\mathcal{E}_K$ with $\mathrm{gr}(\mathcal{E}^\bullet_K) \in \mathscr{U}(K)$, by Lemma \ref{dec} $\mathrm{gr}(\mathcal{E}^\bullet_K)=\mathcal{L}_K \oplus \mathcal{F}_K \in \mathscr{U}(K)$ for some $\mathcal{L}_K \in \mathscr{B}un_1^1(K)$ and $\mathcal{F}_K \in \mathscr{B}un_2^{1}(K)$. Then the filtration $\mathcal{E}_K^\bullet$ is of the form
\[
0 \to \mathcal{L}_K \to \mathcal{E}_K \to \mathcal{F}_K \to 0 \text{ or } 0 \to \mathcal{F}_K \to \mathcal{E}_K \to \mathcal{L}_K \to 0.
\]
In the first case, the uniquely induced filtration $\mathcal{E}_\kappa^\bullet$ of $\mathcal{E}_\kappa$ is of the form
\[
0 \to \mathcal{L}_\kappa \to \mathcal{E}_\kappa \to \mathcal{F}_\kappa \to 0
\]
and we need to show $\mathrm{gr}(\mathcal{E}^\bullet_\kappa)=\mathcal{L}_\kappa \oplus \mathcal{F}_\kappa \in \mathscr{U}(\kappa)$. Note that $\mathcal{L}_\kappa \in \mathscr{B}un_1^1(\kappa)$ and it is a subbundle of $\mathcal{E}_\kappa$ by Lemma \ref{no}, i.e., $\mathcal{F}_\kappa \in \mathscr{B}un_2^1(\kappa)$. Since $\mathcal{E}_\kappa \in \mathscr{U}(\kappa)$, we are done by Corollary \ref{d}. 
The second case is similar.
\end{proof}
To verify S-completeness we need to know points in $\mathscr{U}$ with opposite filtrations, which can be completely determined as follows.
\begin{lem}\label{fil}
Let $\mathcal{E} \ncong \mathcal{E}' \in \mathscr{U}(\mathbf{K})$ for some field $\mathbf{K}/k$ be two points with opposite filtrations $\mathcal{E}^\bullet,\mathcal{E}'_\bullet$. Assume that $\mathcal{E}$ is unstable with Harder-Narasimhan filtration $0 \to \mathcal{L} \to \mathcal{E} \to \mathcal{F} \to 0$. Then the underlying unweighted filtrations of $\mathcal{E}^\bullet,\mathcal{E}'_\bullet$ take one of the following forms:
\begin{enumerate}
\item
$\mathcal{E}^\bullet:0 \subseteq \mathcal{L} \subseteq \mathcal{E}$ and $\mathcal{E}'_\bullet: \mathcal{E}' \supseteq \mathcal{F} \supseteq 0$.
\item
$\mathcal{E}^\bullet: 0 \subseteq \mathcal{F} \subseteq \mathcal{E}$ and $\mathcal{E}'_\bullet: \mathcal{E}' \supseteq \mathcal{L} \supseteq 0$.
\end{enumerate}
\end{lem}
\begin{proof}
Let $a \in \mathbf{Z}$ be the minimal integer such that $\mathcal{E}^a=\mathcal{E}$, which is also the minimal integer such that $\mathcal{E}'_{a+1}=0$. Let $b \in \mathbf{Z}$ be the maximal integer such that $\mathcal{E}'_b=\mathcal{E}'$, which is also the maximal integer such that $\mathcal{E}^{b-1}=0$. Assume $\mathcal{E}^{i-1} \neq \mathcal{E}^{i}$ for $a \geq i \geq b$. By definition we have $\mathrm{rk}(\mathcal{E}/\mathcal{E}^{a-1})=\mathrm{rk}(\mathcal{E}'_a) \geq 1$ and $\mathrm{rk}(\mathcal{E}'/\mathcal{E}'_{b+1})=\mathrm{rk}(\mathcal{E}^b) \geq 1$. Thus,
\[
2 \geq \mathrm{rk}(\mathcal{E}^{a-1}) \geq \cdots \geq \mathrm{rk}(\mathcal{E}^b) \geq 1 \text{ and } 2 \geq \mathrm{rk}(\mathcal{E}'_{b+1}) \geq \cdots \geq \mathrm{rk}(\mathcal{E}'_a) \geq 1.
\]
Applying Lemma \ref{no} to $\mathcal{E}$ and $\mathcal{E}'$ yields
\[
\deg(\mathcal{E}^i) \leq 1 \text{ and } \deg(\mathcal{E}'_{i+1}) \leq 1 \text{ for } a-1 \geq i \geq b.
\]
However, the isomorphism $\mathcal{E}^i/\mathcal{E}^{i-1} \cong \mathcal{E}'_i/\mathcal{E}'_{i+1}$ implies
\[
\deg(\mathcal{E}^i)+\deg(\mathcal{E}'_{i+1})=2 \text{ for } a-1 \geq i \geq b.
\]
Then we obtain $\deg(\mathcal{E}^i)=\deg(\mathcal{E}'_{i+1})=1$. By Lemma \ref{no} this means each $\mathcal{E}^i$ (resp., $\mathcal{E}'_{i+1}$) is a subbundle of $\mathcal{E}$ (resp., $\mathcal{E}'$) and thus
\[
2 \geq \mathrm{rk}(\mathcal{E}^{a-1}) > \cdots > \mathrm{rk}(\mathcal{E}^b) \geq 1 \text{ and } 2 \geq \mathrm{rk}(\mathcal{E}'_{b+1}) > \cdots > \mathrm{rk}(\mathcal{E}'_a) \geq 1,
\]
In particular, $2 \geq a-b \geq 1$. But $a-b=2$ is impossible as in any filtration $\mathcal{E}^\bullet: 0 \subseteq \mathcal{E}^b \subseteq \mathcal{E}^{a-1} \subseteq \mathcal{E}$ of length $3$ we have $\mathcal{E}^{a-1}$ is $(1,0)$-stable by Corollary \ref{d}. This is a contradiction since $\mathcal{E}^b \subseteq\mathcal{E}^{a-1}$ is a degree $1$ sub-line bundle. Then $a-b=1$, i.e.,
\[
\mathcal{E}^\bullet: 0 \subseteq \mathcal{E}^b \subseteq \mathcal{E} \text{ and } \mathcal{E}'_\bullet: \mathcal{E}' \supseteq \mathcal{E}'_{b+1} \supseteq 0.
\]
There are 2 cases: either $\mathrm{rk}(\mathcal{E}^b)=1$ or $\mathrm{rk}(\mathcal{E}^b)=2$. If $\mathrm{rk}(\mathcal{E}^b)=1$, then we have $\mathcal{E}^b \cong \mathcal{L}$ by the uniqueness of Harder-Narasimhan filtration and $\mathcal{E}'_{b+1} \cong \mathcal{E}/\mathcal{E}^b \cong \mathcal{E}/\mathcal{L}=\mathcal{F}$. This is case (i). If $\mathrm{rk}(\mathcal{E}^b)=2$, then we know $\mathcal{E}^b$ is stable by Corollary \ref{d} and the non-zero composition $\mathcal{E}^b \hookrightarrow \mathcal{E} \twoheadrightarrow \mathcal{F}$ is an isomorphism, splitting the sequence $0 \to \mathcal{L} \to \mathcal{E} \to \mathcal{F} \to 0$. Thus, $\mathcal{E}'_{b+1} \cong \mathcal{E}/\mathcal{E}^b \cong \mathcal{E}/\mathcal{F}=\mathcal{L}$. This is case (ii).
\end{proof}
\begin{cor}\label{S-comp}
The open substack $\mathscr{U}$ is S-complete.
\end{cor}
\begin{proof}
Using Proposition \ref{prop:S-comple}. For any field $\mathbf{K}/k$ and any two points $\mathcal{E} \ncong \mathcal{E}' \in \mathscr{U}(\mathbf{K})$ with opposite filtrations $\mathcal{E}^\bullet,\mathcal{E}'_\bullet$, we may assume that $\mathcal{E}$ is unstable with Harder-Narasimhan filtration $0 \to \mathcal{L} \to \mathcal{E} \to \mathcal{F} \to 0$. By Lemma \ref{fil}, the underlying unweighted filtrations of such $\mathcal{E}^\bullet,\mathcal{E}'_\bullet$ take the form
\[
\begin{cases}
\mathcal{E}^\bullet: & 0 \subseteq\mathcal{F} \subseteq\mathcal{E} \\
\mathcal{E}'_\bullet: & 0 \subseteq\mathcal{L} \subseteq\mathcal{E}'
\end{cases}
\text{ or }
\begin{cases}
\mathcal{E}^\bullet: & 0 \subseteq\mathcal{L} \subseteq\mathcal{E} \\
\mathcal{E}'_\bullet: & 0 \subseteq\mathcal{F} \subseteq\mathcal{E}'.
\end{cases}
\]
In either case, $\mathcal{L} \oplus \mathcal{F} \in \mathscr{U}(\mathbf{K})$ by Corollary \ref{d}. This finishes the proof.
\end{proof}
Finally, the open substack $\mathscr{U}$ does not satisfy the existence part of the valuative criterion for properness since there are unextendable families.
\begin{lem}\label{nProper}
There exists a family $\mathcal{E}_R \in \mathscr{B}un_3^{2,ss}(R)$ for some DVR $R$ over $k$ with fraction field $K$ and residue field $\kappa$ such that
\begin{enumerate}
\item
the generic fiber $\mathcal{E}_K$ fits into a short exact sequence
\[
0 \to \mathcal{F}_K \to \mathcal{E}_K \to \mathcal{L}_K \to 0
\]
for some $\mathcal{L}_K \in \mathscr{B}un_1^1(K)$ and $\mathcal{F}_K \in \mathscr{B}un_2^1(K)$ with $\mathcal{F}_K$ being $(1,0)$-stable. In particular, $\mathcal{E}_K \in \mathscr{U}(K)$.
\item
the special fiber $\mathcal{E}_\kappa$ fits into a short exact sequence
\[
0 \to \mathcal{F}_\kappa \to \mathcal{E}_\kappa \to \mathcal{L}_\kappa \to 0
\]
for some $\mathcal{L}_\kappa \in \mathscr{B}un_1^1(\kappa)$ and $\mathcal{F}_\kappa \in \mathscr{B}un_2^1(\kappa)$ with $\mathcal{F}_\kappa$ being stable but not $(1,0)$-stable. In particular, $\mathcal{E}_\kappa \notin \mathscr{U}(\kappa)$.
\end{enumerate}
\end{lem}
\begin{proof}
Since $\mathscr{B}un_2^{1,s}$ is irreducible, there exists a family $\mathcal{F}_R \in \mathscr{B}un_2^{1,s}(R)$ for some DVR $R$ over $k$ with fraction field $K$ and residue field $\kappa$ such that 
\[
\mathcal{F}_K \in \mathscr{B}un_2^{1,s}(1,0)(K) \text{ and } \mathcal{F}_\kappa \in \left(\mathscr{B}un_2^{1,s}-\mathscr{B}un_2^{1,s}(1,0)\right)(\kappa).
\]
Fix a family $\mathcal{L}_R \in \mathscr{B}un_1^1(R)$ and a non-split extension $[\varrho]: 0 \to \mathcal{F}_\kappa \to \mathcal{E}_\kappa \to \mathcal{L}_\kappa \to 0$, the same argument as in Lemma \ref{dec} shows that $\mathcal{E}_\kappa$ is stable. Let $\mathcal{E}_{R} \in \mathscr{C}oh_3^2(R)$ be an extension family of $\mathcal{L}_R$ by $\mathcal{F}_R$ with special fiber $[\varrho]$ (see Lemma \ref{2321}). Since $\mathcal{E}_\kappa \in \mathscr{B}un_3^{2,ss}(\kappa)$, by the openness of stability we have $\mathcal{E}_{R} \in \mathscr{B}un_3^{2,ss}(R)$.
\end{proof}
\begin{cor}\label{nPro}
The open substack $\mathscr{U}$ does not satisfy the existence part of the valuative criterion for properness.
\end{cor}
\begin{proof}
Any family $\mathcal{E}_R \in \mathscr{B}un_3^{2,ss}(R)$ as in Lemma \ref{nProper} defines a diagram
\[
\begin{tikzcd}
\mathrm{Spec}(K) \ar[r,"{\mathcal{E}_K}"] \ar[d,hook] &  \mathscr{U} \ar[d] \\
\mathrm{Spec}(R) \ar[ur,dashed,"\times"'] \ar[r] & \mathrm{Spec}(k)
\end{tikzcd}
\]
of solid arrows and we claim that there is no dotted arrow filling it in. Assume otherwise that there exists a family $\mathcal{E}'_R \in \mathscr{U}(R)$ such that $\mathcal{E}'_K \cong \mathcal{E}_K$. Then we consider the short exact sequences $0 \to \mathcal{F}'_{\kappa} \to \mathcal{E}'_{\kappa} \to \mathcal{L}'_{\kappa}\to 0$ induced by that on the generic fiber $0 \to \mathcal{F}_K \to \mathcal{E}_K \to \mathcal{L}_K \to 0$ in Lemma \ref{nProper}. Since $\mathcal{E}'_\kappa \in \mathscr{U}(\kappa)$, we have $\mathcal{F}'_\kappa$ is $(1,0)$-stable by Corollary \ref{d}. Both $\mathcal{F}_\kappa$ and $\mathcal{F}'_\kappa$ are stable limits of $\mathcal{F}_K$, so they are isomorphic by Langton's theorem A (see \cite{MR0364255}). But $\mathcal{F}'_\kappa$ is $(1,0)$-stable while $\mathcal{F}_\kappa$ is not, a contradiction.
\end{proof}
\begin{proof}[Proof of Theorem $\ref{rank-3-main}$]
The open substack $\mathscr{U} \subseteq \mathscr{B}un_3^{2}$ is non-empty and not contained in $\mathscr{B}un_3^{2,ss}$ by Lemma \ref{dec}. It is locally reductive (Lemma \ref{local-theta}), $\Theta$-reductive (Lemma \ref{theta}), S-complete (Corollary \ref{S-comp}), and does not satisfy the existence part of the valuative criterion (Corollary \ref{nPro}). By Theorem \ref{thm0}, $\mathscr{U}$ admits a separated non-proper adequate moduli space. Furthermore, since every closed point of $\mathscr{U}$ has linearly reductive stabilizer by the proof of Lemma \ref{local-theta}, this adequate moduli space is actually a good moduli space by \cite[Corollary 6.11]{Alper2019-1}.
\end{proof}
\begin{rmk}
It is also interesting to find a compactification of $\mathscr{U} \subseteq \mathscr{B}un_3^2$, i.e., an open substack $\overline{\mathscr{U}} \subseteq \mathscr{B}un_3^2$ containing $\mathscr{U}$ that admits a proper good moduli space. The na\"ive candidate $\mathscr{U} \cup \mathscr{B}un_3^{2,ss}$, while S-complete, is not $\Theta$-reductive.
\end{rmk}
\begin{acknowledgements}
Part of the work presented in this paper constitutes Xucheng Zhang's thesis carried out at Universit\"{a}t Duisburg-Essen. He would like to thank his supervisor Jochen Heinloth for suggesting this topic and sharing many ideas; Mingshuo Zhou for pointing out the literature \cite{MR0508172} and lots of fruitful discussions; and Bin Wang for explaining concepts and checking computations on reductive groups. Dario Wei\ss mann would like to thank his supervisor Georg Hein for his support and teaching him about the moduli space of vector bundles as well as his second supervisor Jochen Heinloth for teaching him about stacks and moduli. He would also like to thank Ludvig Modin for a discussion on adequate moduli.
Both authors would like to thank Jarod Alper for many inspiring discussions. 
We thank the anonymous referee for many thoughtful corrections and insightful suggestions which allowed us to generalize Theorem \ref{1535-2} to reductive groups and remove the characteristic 0 restriction in a previous version of Theorem \ref{thmA} and \ref{thmB}. 
In particular, Remark \ref{3de}, Proposition \ref{lem:localreductive}, and Remark \ref{1505-5} are due to the referee's sharing.
\end{acknowledgements}

\bibliographystyle{amsalpha}
\bibliography{Thesis-ref}

\newcommand{\etalchar}[1]{$^{#1}$}
\providecommand{\bysame}{\leavevmode\hbox to3em{\hrulefill}\thinspace}
\providecommand{\MR}{\relax\ifhmode\unskip\space\fi MR }
\providecommand{\MRhref}[2]{%
  \href{http://www.ams.org/mathscinet-getitem?mr=#1}{#2}
}
\providecommand{\href}[2]{#2}
\begin{thebibliography}{ABB{\etalchar{+}}22}

\bibitem[ABB{\etalchar{+}}22]{MR4480534}
Jarod Alper, Pieter Belmans, Daniel Bragg, Jason Liang, and Tuomas Tajakka, \emph{Projectivity of the moduli space of vector bundles on a curve}, Stacks {P}roject {E}xpository {C}ollection ({SPEC}), London Math. Soc. Lecture Note Ser., vol. 480, Cambridge Univ. Press, Cambridge, 2022, pp.~90--125. \MR{4480534}

\bibitem[ACV03]{MR2007376}
Dan Abramovich, Alessio Corti, and Angelo Vistoli, \emph{Twisted bundles and admissible covers}, Comm. Algebra \textbf{31} (2003), no.~8, 3547--3618, Special issue in honor of Steven L. Kleiman. \MR{2007376}

\bibitem[AHLH23]{MR4665776}
Jarod Alper, Daniel Halpern-Leistner, and Jochen Heinloth, \emph{Existence of moduli spaces for algebraic stacks}, Invent. Math. \textbf{234} (2023), no.~3, 949--1038. \MR{4665776}

\bibitem[AHR19]{Alper2019-1}
Jarod Alper, Jack Hall, and David Rydh, \emph{The \'{e}tale local structure of algebraic stacks}, \href{https://arxiv.org/abs/1912.06162}{arXiv: 1912.06162}, December 2019.

\bibitem[AHR20]{MR4088350}
\bysame, \emph{A {L}una \'{e}tale slice theorem for algebraic stacks}, Ann. of Math. (2) \textbf{191} (2020), no.~3, 675--738. \MR{4088350}

\bibitem[Alp13]{MR3237451}
Jarod Alper, \emph{Good moduli spaces for {A}rtin stacks}, Ann. Inst. Fourier (Grenoble) \textbf{63} (2013), no.~6, 2349--2402. \MR{3237451}

\bibitem[Alp14]{MR3272912}
\bysame, \emph{Adequate moduli spaces and geometrically reductive group schemes}, Algebr. Geom. \textbf{1} (2014), no.~4, 489--531. \MR{3272912}

\bibitem[BH10]{MR2628848}
Indranil Biswas and Norbert Hoffmann, \emph{The line bundles on moduli stacks of principal bundles on a curve}, Doc. Math. \textbf{15} (2010), 35--72. \MR{2628848}

\bibitem[DS95]{MR1362973}
V.~G. Drinfeld and Carlos Simpson, \emph{{$B$}-structures on {$G$}-bundles and local triviality}, Math. Res. Lett. \textbf{2} (1995), no.~6, 823--829. \MR{1362973}

\bibitem[Fal03]{MR1961134}
Gerd Faltings, \emph{Algebraic loop groups and moduli spaces of bundles}, J. Eur. Math. Soc. (JEMS) \textbf{5} (2003), no.~1, 41--68. \MR{1961134}

\bibitem[GLSS08]{MR2450609}
T.~L. G\'{o}mez, A.~Langer, A.~H.~W. Schmitt, and I.~Sols, \emph{Moduli spaces for principal bundles in arbitrary characteristic}, Adv. Math. \textbf{219} (2008), no.~4, 1177--1245. \MR{2450609}

\bibitem[GPHS14]{MR3293805}
Oscar Garc\'{\i}a-Prada, Jochen Heinloth, and Alexander Schmitt, \emph{On the motives of moduli of chains and {H}iggs bundles}, J. Eur. Math. Soc. (JEMS) \textbf{16} (2014), no.~12, 2617--2668. \MR{3293805}

\bibitem[Hei17]{MR3758902}
Jochen Heinloth, \emph{Hilbert-{M}umford stability on algebraic stacks and applications to {$\mathcal{G}$}-bundles on curves}, \'{E}pijournal G\'{e}om. Alg\'{e}brique \textbf{1} (2017), Art. 11, 37. \MR{3758902}

\bibitem[HL10]{MR2665168}
Daniel Huybrechts and Manfred Lehn, \emph{The geometry of moduli spaces of sheaves}, second ed., Cambridge Mathematical Library, Cambridge University Press, Cambridge, 2010. \MR{2665168}

\bibitem[HL14]{DHL2014}
Daniel Halpern-Leistner, \emph{On the structure of instability in moduli theory}, \href{https://arxiv.org/abs/1411.0627}{arXiv: 1411.0627}, November 2014.

\bibitem[HLH23]{DA23}
Daniel Halpern-Leistner and Andres~Fernandez Herrero, \emph{The structure of the moduli of gauged maps from a smooth curve}, \href{https://arxiv.org/abs/2305.09632}{arXiv: 2305.09632}, May 2023.

\bibitem[Hof10a]{MR2605167}
Norbert Hoffmann, \emph{Moduli stacks of vector bundles on curves and the {K}ing-{S}chofield rationality proof}, Cohomological and geometric approaches to rationality problems, Progr. Math., vol. 282, Birkh\"{a}user Boston, Inc., Boston, MA, 2010, pp.~133--148. \MR{2605167}

\bibitem[Hof10b]{MR3013030}
\bysame, \emph{On moduli stacks of {$G$}-bundles over a curve}, Affine flag manifolds and principal bundles, Trends Math., Birkh\"{a}user/Springer Basel AG, Basel, 2010, pp.~155--163. \MR{3013030}

\bibitem[Hof12]{MR2925604}
\bysame, \emph{The {P}icard group of a coarse moduli space of vector bundles in positive characteristic}, Cent. Eur. J. Math. \textbf{10} (2012), no.~4, 1306--1313. \MR{2925604}

\bibitem[HR19]{MR4009673}
Jack Hall and David Rydh, \emph{Coherent {T}annaka duality and algebraicity of {H}om-stacks}, Algebra Number Theory \textbf{13} (2019), no.~7, 1633--1675. \MR{4009673}

\bibitem[HS10]{MR2657374}
Jochen Heinloth and Alexander H.~W. Schmitt, \emph{The cohomology rings of moduli stacks of principal bundles over curves}, Doc. Math. \textbf{15} (2010), 423--488. \MR{2657374}

\bibitem[Lan75]{MR0364255}
Stacy~G. Langton, \emph{Valuative criteria for families of vector bundles on algebraic varieties}, Ann. of Math. (2) \textbf{101} (1975), 88--110. \MR{0364255}

\bibitem[LO87]{MR0879544}
M.~L\"{u}bke and C.~Okonek, \emph{Moduli spaces of simple bundles and {H}ermitian-{E}instein connections}, Math. Ann. \textbf{276} (1987), no.~4, 663--674. \MR{879544}

\bibitem[MS09]{MR250496}
Ngaiming Mok and Xiaotao Sun, \emph{Remarks on lines and minimal rational curves}, Sci. China Ser. A \textbf{52} (2009), no.~4, 617--630. \MR{2504964}

\bibitem[Nor78]{MR0508172}
V.~Alan Norton, \emph{Nonseparation in the moduli of complex vector bundles}, Math. Ann. \textbf{235} (1978), no.~1, 1--16. \MR{0508172}

\bibitem[Nor79]{MR0529671}
\bysame, \emph{Analytic moduli of complex vector bundles}, Indiana Univ. Math. J. \textbf{28} (1979), no.~3, 365--387. \MR{529671}

\bibitem[NR78]{MR541029}
M.~S. Narasimhan and S.~Ramanan, \emph{Geometry of {H}ecke cycles. {I}}, C. {P}. {R}amanujam---a tribute, Tata Inst. Fund. Res. Studies in Math., vol.~8, Springer, Berlin-New York, 1978, pp.~291--345. \MR{541029}

\bibitem[NS65]{MR0184252}
M.~S. Narasimhan and C.~S. Seshadri, \emph{Stable and unitary vector bundles on a compact {R}iemann surface}, Ann. of Math. (2) \textbf{82} (1965), 540--567. \MR{0184252}

\bibitem[{Sta}24]{stacks-project}
The {Stacks Project Authors}, \emph{Stacks project}, \url{https://stacks.math.columbia.edu}, 2024.

\bibitem[Wan11]{Wang2011}
Jonathan Wang, \emph{The moduli stack of ${G}$-bundles}, \href{https://arxiv.org/abs/1104.4828}{arXiv: 1104.4828}, April 2011.

\bibitem[Zha22]{thesis-zhang}
Xucheng Zhang, \emph{Characterizing open substacks of algebraic stacks that admit good moduli spaces}, Ph.D. thesis, Universität Duisburg-Essen, Sep 2022, 150 pp. \href{https://doi.org/10.17185/duepublico/76595}{https://doi.org/10.17185/duepublico/76595}.

\end{thebibliography}

\end{document}